\documentclass[reqno]{amsart}
\usepackage{amssymb}
\usepackage[utf8]{inputenc}
\usepackage{amsmath,amsthm,amstext,amssymb,bbm}
\usepackage{amscd,  amsfonts,  amsbsy,  mathrsfs, upgreek, mathtools, stmaryrd, bbm}
\usepackage{todonotes} 
\usepackage{pdfsync}
\usepackage{hyperref}

\usepackage{geometry}

\newtheorem{theorem}{Theorem}[section]
\newtheorem{proposition}[theorem]{Proposition}
\newtheorem{lemma}[theorem]{Lemma}
\newtheorem{corollary}[theorem]{Corollary}
\newtheorem{question}[theorem]{Question}

\newtheorem*{claim*}{Claim}
\newtheorem*{subclaim*}{Subclaim}
\theoremstyle{definition}
\newtheorem{remark}[theorem]{Remark}
\newtheorem{definition}[theorem]{Definition}

\newcommand{\Ce}{{\mathcal{C}}}

\newcommand{\VV}{{\rm{V}}}

\newcommand{\PPP}{{\mathbb{P}}}
\newcommand{\QQQ}{{\mathbb{Q}}}
\newcommand{\RRR}{{\mathbb{R}}}

\newcommand{\aaaa}{{\mathfrak{a}}}
\newcommand{\FFFF}{{\mathfrak{F}}}
\newcommand{\CCCC}{{\mathfrak{C}}}
\newcommand{\RRRR}{{\mathfrak{R}}}
\newcommand{\SSSS}{{\mathfrak{S}}}

\newcommand{\HH}[1]{{\rm{H}}(#1)}
\newcommand{\dom}[1]{{{\rm{dom}}(#1)}}

\newcommand{\POT}[1]{{\mathcal{P}}({#1})}

\newcommand{\seq}[2]{\langle{#1}~\vert~{#2}\rangle}
\newcommand{\map}[3]{{#1}:{#2}\longrightarrow{#3}}

\newcommand{\ran}[1]{{{\rm{ran}}(#1)}}
\newcommand{\id}{{\rm{id}}}
\newcommand{\crit}[1]{{{\rm{crit}}\left({#1}\right)}}
\newcommand{\calL}{\mathcal{L}}

\newcommand{\cof}[1]{{{\rm{cof}}(#1)}}
\newcommand{\Set}[2]{\{{#1}~ \vert~{#2}\}}

\newcommand{\anf}[1]{{\text{``}\hspace{0.3ex}{#1}\hspace{0.3ex}\text{''}}}

\newcommand{\LL}{{\rm{L}}}
\newcommand{\HOD}{{\rm{HOD}}}
\newcommand{\ZFC}{{\rm{ZFC}}}
\newcommand{\ZF}{{\rm{ZF}}}
\newcommand{\Ord}{{\rm{Ord}}}
\newcommand{\Lim}{{\rm{Lim}}}

\newcommand{\SCH}{{\rm{SCH}}}

\newcommand{\betrag}[1]{\vert{#1}\vert}

\newcommand{\calA}{\mathcal{A}}

\newcommand{\goedel}[2]{{\prec}{#1},{#2}{\succ}}
\newcommand{\Add}[2]{{\rm{Add}}({#1},{#2})}
\newcommand{\Str}{Str}
\newcommand{\ooo}{o}

\author{Philipp L\"{u}cke}
\address{Philipp L\"{u}cke. Fachbereich Mathematik, Universit\"at Hamburg, Bundesstra{\ss}e 55, Hamburg, 20146, Germany}
\email{philipp.luecke@uni-hamburg.de}

\makeatletter
\@namedef{subjclassname@2020}{\textup{2020} Mathematics Subject Classification}
\makeatother

\title[Weak compactness cardinals and subtlety properties]{Weak compactness cardinals for strong logics and subtlety properties of the class of ordinals}


\thanks{The author would like to thank Will Boney and Victoria Gitman  for discussions that motivated much of the content of this paper. 
In addition, he would like to thank Toshimichi Usuba for helpful comments on earlier versions of the presented results. 
Finally, the author would like to thank the anonymous referee for a very thorough reading of the manuscript and numerous helpful comments.
 The  author gratefully acknowledges supported by the Deutsche Forschungsgemeinschaft (Project number 522490605).}

\begin{document}

\begin{abstract}
 Motivated by recent work of Boney, Dimopoulos, Gitman and Magidor, we characterize the existence of weak compactness cardinals for all abstract logics  through combinatorial properties of the class of ordinals. 
 This analysis is then used to show that, in contrast to the existence of strong compactness cardinals, 
  the existence of weak compactness cardinals for  abstract logics does not  imply the existence of a strongly inaccessible cardinal. 
  More precisely, it is proven that the existence of a proper class of subtle cardinals is consistent with the axioms of $\ZFC$ if and only if it is not possible to derive the existence of strongly inaccessible cardinals from the existence of weak compactness cardinals for all abstract logics.  
  Complementing this result, it is shown that the existence of weak compactness cardinals for all abstract logics implies that unboundedly many ordinals are strongly inaccessible in the inner model $\HOD$ of all hereditarily ordinal definable sets. 
\end{abstract}

\maketitle


\section{Introduction}

The theory of strong axioms of infinity and the study of extensions of first-order logic are deeply connected through results showing that  
 the ability to generalize fundamental structural features of first-order logic to stronger logics 
 is equivalent to the existence of   large cardinals.\footnote{Throughout this paper,  we refer to properties of cardinal numbers that imply weak inaccessibility as \emph{large cardinal properties}.} 
%
%
%
%
 A classical example of such a result is  a theorem of Magidor in \cite{MR295904} that shows that the existence of 
 a \emph{strong compactness cardinal} for second-order logic $\calL^2$ ({i.e.,} the existence of a cardinal $\kappa$ with the property that every unsatisfiable $\calL^2$-theory contains an unsatisfiable subtheory of cardinality less than $\kappa$) 
 is equivalent to the existence of an extendible cardinal.  
 In \cite{MR780522}, Makowsky proved an analog of this result that provides a large cardinal characterization of the existence of strong compactness cardinals for all abstract logics.\footnote{The definition of an abstract logic and all related model-theoretic notions can be found in Section \ref{section:AbstractLogics}.}
  Recall that \emph{Vop\v{e}nka's Principle} is the scheme of axioms stating that for every proper class\footnote{Throughout this paper, we work in $\ZFC$. Therefore,  all classes are definable by first-order formulas in the language of set theory, possibly using sets as parameters.} of graphs, there are two distinct members of the class with a homomorphism between them. 
 The validity of this combinatorial principle is known to be characterizable through the existence of certain large cardinals (see {\cite[Section 4]{zbMATH06029795}} and {\cite[Section 6]{MR482431}}). 
 Makowsky's result  now states the following:

 \begin{theorem}[{\cite[Theorem 2]{MR780522}}]\label{theorem:Makowsky}
  The following schemes are equivalent over $\ZFC$: 
   \begin{enumerate}
   \item Vop\v{e}nka's Principle. 
   
   \item Every abstract logic has a strong compactness cardinal. 
   \end{enumerate}
 \end{theorem}

 The work presented in this paper is motivated by recent results of Boney, Dimopoulos, Gitman and Magidor in \cite{bdgm} that deal with the existence of \emph{weak compactness cardinals} for abstract logics, {i.e.,} the existence of cardinals $\kappa$ with the property that every unsatisfiable theory of cardinality $\kappa$ contains an unsatisfiable subtheory of cardinality less than $\kappa$. 
 These results show that the existence of weak compactness cardinals for strong logics is connected to the combinatorial concept of \emph{subtleness}, introduced by Jensen and Kunen in \cite{jensennotes}. Remember that an infinite cardinal $\kappa$ is \emph{subtle} if for every sequence $\seq{A_\alpha}{\alpha<\kappa}$ with $A_\alpha\subseteq\alpha$ for all $\alpha<\kappa$ and every closed unbounded subset $C$ of $\kappa$, there are $\alpha,\beta\in C$ with $\alpha<\beta$ and $A_\beta\cap\alpha=A_\alpha$. A short argument shows that subtle cardinals are inaccessible limits of totally indescribable cardinals. 
 This large cardinal notion  has a canonical class-version, {i.e.,} one defines \anf{\emph{$\Ord$ is subtle}} to be the scheme of sentences  stating that for every class sequence $\seq{A_\alpha}{\alpha\in\Ord}$ with $A_\alpha\subseteq\alpha$ for all $\alpha\in\Ord$ and every closed unbounded class $C$ of ordinals, there exist $\alpha,\beta\in C$ with $\alpha<\beta$ and $A_\beta\cap\alpha=A_\alpha$. 
 The results of \cite{bdgm} now show that, in the presence of a definable well-ordering of $\VV$,\footnote{Note that the existence of such a well-ordering is equivalent to the statement that $\VV=\HOD_x$ for some set $x$ and can therefore be formulated as a single sentence in the language of set theory.} the validity of this principle is connected to the existence of weak compactness cardinals for abstract logics:

 \begin{theorem}[{\cite[Theorem 5.3]{bdgm}}]\label{theorem:BDGM}
  The following schemes are equivalent over $\ZFC$ together with the existence of a definable well-ordering of $\VV$: 
  \begin{enumerate}
   \item $Ord$ is subtle. 
   
   \item Every abstract logic has a stationary class of weak compactness cardinals. 
  \end{enumerate}
 \end{theorem}

 This result immediately raises the following questions: 
 \begin{itemize}
  \item Is it also possible to characterize the mere existence of weak compactness cardinals for every abstract logic (instead of the existence of stationary classes of such cardinals) through combinatorial properties of the class of ordinals? 
  
  \item Is it  possible to establish such characterizations in the absence of definable well-orderings of $\VV$? 
 \end{itemize}

 The results of this paper will provide affirmative answers to both of these questions. 
 Moreover, this analysis will be based on arguments  showing that the validity of the involved combinatorial principles for the class of ordinals is equivalent to the existence of  certain large cardinals.

 First, in order to prove results that do not rely on the existence of definable well-orderings, we will make use of the following variation of the above class-principle that was introduced by Bagaria and the author in their work on \emph{principles of structural reflection}:

\begin{definition}[\cite{Patterns}]
 We let \anf{\emph{$\Ord$ is essentially subtle}} denote the scheme of sentences stating that for every class sequence $\seq{E_\alpha}{\alpha\in\Ord}$ with $\emptyset\neq E_\alpha\subseteq\POT{\alpha}$  for all $\alpha\in\Ord$ and  every closed unbounded class $C$ of ordinals, there exist $\alpha,\beta\in C$ and  $A\in E_\beta$ with $\alpha<\beta$  and $A\cap\alpha\in E_\alpha$.  
\end{definition}

Obviously, if $\Ord$ is essentially subtle, then $\Ord$ is subtle. Moreover, in the presence of a definable well-ordering of $\VV$, these principles are equivalent. In particular, since the assumption that $\Ord$ is subtle is easily seen to be downwards absolute from $\VV$ to the constructible universe $\LL$, it follows that these two principles are equiconsistent over $\ZFC$. 
 In contrast, by combining a result in \cite{bdgm}  with Theorem \ref{MAIN:EssentiallySubtle} below, it is possible to show that, over $\ZFC$, the assumption that  $\Ord$ is subtle does not imply that $\Ord$ is essentially subtle (see Corollary \ref{corollary:SublteNotEssentially} below). 
 The following direct strengthening of Theorem \ref{theorem:BDGM}, proven in Section \ref{section:ProofsCompactness} below,  now provides the desired  $\ZFC$-characterization of the existence of stationary classes of weak compactness cardinals for all abstract logics:

\begin{theorem}\label{MAIN:EssentiallySubtle}
 The following schemes of sentences are equivalent over $\ZFC$: 
 \begin{enumerate}
  \item $\Ord$ is essentially subtle. 
  
  \item Every abstract logic has a stationary class of weak compactness cardinals. 
  
 \end{enumerate}
\end{theorem}

 Next, we present a  result  showing that the statement that every abstract logic has at least one weak compactness cardinal 
 is also equivalent to a combinatorial property of the class $\Ord$. 
 %
 %
 This property generalizes the notion of \emph{faintness}, introduced by Matet in \cite{MR930262}, from cardinals to the class $\Ord$. Matet defined a cardinal $\kappa$ to be \emph{faint} if for every sequence $\seq{A_\alpha}{\alpha<\kappa}$ with $A_\alpha\subseteq\alpha$ for all $\alpha<\kappa$ and every $\xi<\kappa$, there are $\xi<\alpha<\beta<\kappa$ with $A_\beta\cap\alpha=A_\alpha$. 
 The results of \cite{MR930262} then show that a cardinal is faint if and only if it is either subtle or a limit of subtle cardinals. 
 Following earlier patterns, we introduce the following class-version of faintness:

\begin{definition}
 We let \anf{\emph{$\Ord$ is essentially faint}} denote the scheme of sentences stating that for every class sequence $\seq{E_\alpha}{\alpha\in\Ord}$ with $\emptyset\neq E_\alpha\subseteq\POT{\alpha}$  for all $\alpha\in\Ord$ and every ordinal $\xi$, there are ordinals $\xi<\alpha<\beta$ and $A\in E_\beta$ with $A\cap\alpha\in E_\alpha$.  
\end{definition}

By definition, if $\Ord$ is essentially subtle, then $\Ord$ is essentially faint. Moreover, it is easy to see that the existence of a proper class of subtle cardinals implies that $\Ord$ is essentially faint:  
 Given a class sequence $\seq{E_\alpha}{\alpha\in\Ord}$ with $\emptyset\neq E_\alpha\subseteq\POT{\alpha}$  for all $\alpha\in\Ord$ and an ordinal $\xi$, the given assumption allows us to  pick a subtle cardinal $\kappa>\xi$ and we can then use the Axiom of Choice to find a sequence $\seq{A_\alpha}{\alpha<\kappa}$ with $A_\alpha\in E_\alpha$ for all $\alpha<\kappa$. 
 The subtleness of $\kappa$ then ensures that the closed unbounded subset $(\xi,\kappa)$ of $\kappa$ contains ordinals $\alpha<\beta$ with $A_\alpha=A_\beta\cap\alpha$. In particular, it follows that there are ordinals $\xi<\alpha<\beta$ and $A\in E_\beta$ with $A\cap\alpha\in E_\alpha$.

 The following result, again proven in Section \ref{section:ProofsCompactness}, shows that this principle exactly determines the  combinatorial property of  $\Ord$ needed for the desired characterization:

\begin{theorem}\label{MAIN:EssentiallyFaint}
 The following schemes of sentences are equivalent over $\ZFC$: 
 \begin{enumerate}
  \item $\Ord$ is essentially faint. 
  
  \item Every abstract logic has a weak compactness cardinal. 
  
 \end{enumerate}
\end{theorem}

 As  mentioned earlier, the proofs of the above theorems rely heavily on characterizations of the involved combinatorial class-principles through large cardinal assumptions. 
 This analysis, carried out in Sections \ref{section:CnStronglyUnfoldable} and \ref{section:CnWeaklyShrewd} below, can also be used to show that the   ability to generalize other structural properties of first-order logic to all abstract logics can also be characterized through the validity of the studied combinatorial principles for $\Ord$. 
 More specifically, motivated by an unpublished result of Stavi that provides an analog to Theorem \ref{theorem:Makowsky} for the existence of \emph{L\"owenheim--Skolem--Tarski numbers} for abstract logics (see Theorem \ref{theorem:Stavi} below), 
 we will prove analogs of Theorems \ref{MAIN:EssentiallySubtle} and \ref{MAIN:EssentiallyFaint}  
  that characterize the existence of \emph{strict L\"owenheim--Skolem--Tarski numbers}, studied by Bagaria and V\"a\"an\"anen in \cite{MR3519447}, through subtlety properties of $\Ord$ (see Theorem \ref{MAIN:StrictLST}).

 The results presented above naturally raise questions about the  relationship between the principle \anf{\emph{$\Ord$ is essentially subtle}} and the principle \anf{\emph{$\Ord$ is essentially faint}}. Trivially, the former principle implies the latter. However, it is not obvious whether these principles are equivalent.
 Somewhat surprisingly, the next result shows that the inequivalence of the two principles has consistency strength strictly greater than the consistency strength of the individual principles. 
 Moreover, this result  uses the inequivalence of the two principles to derive the existence of certain large cardinals not only in the constructible universe $\LL$, but also in the class $\HOD$ of all hereditarily ordinal definable sets.

\begin{theorem}\label{theorem:SubtleInHOD}
 If  $\Ord$ is essentially faint and only boundedly many cardinals are subtle in $\HOD$, then $\Ord$ is essentially subtle. 
\end{theorem}

In particular, in the unlikely case that the axioms of $\ZFC$ prove that there are only set-many subtle cardinals, the above theorem implies that these axioms also prove that the principle \anf{\emph{$\Ord$ is essentially subtle}} is equivalent to the principle \anf{\emph{$\Ord$ is essentially faint}}. 
Our later results will show that the converse of this implication also holds, {i.e.,} if the axioms of $\ZFC$ are consistent with the existence of a proper class of subtle cardinals, then the given principles are not provably equivalent (see Theorem \ref{MAIN:SeparatePrinciples} below).

Since the existence of a subtle cardinal clearly implies the existence of a set-sized model of $\ZFC$ in which $\Ord$ is essentially subtle, Theorem \ref{theorem:SubtleInHOD} directly yields the following equiconsistency:

\begin{corollary}\label{corollary:EquiconsEsubtleEfaint}
  The following theories are equiconsistent: 
     \begin{enumerate}
      \item $\ZFC + \anf{\textit{$\Ord$ is essentially subtle}}$. 

      \item $\ZFC + \anf{\textit{$\Ord$ is essentially faint}}$.  \qed 
     \end{enumerate} 
\end{corollary}

Motivated by Theorem \ref{theorem:SubtleInHOD}, we now consider the possibility of separating the principle \anf{\emph{$\Ord$ is essentially faint}} from the principle \anf{\emph{$\Ord$ is essentially subtle}}. 
Our results will show that this question is closely related to questions about the existence of sentences in the language of set theory that imply all sentences appearing in the  given schemes of sentences. To motivate these results, we will first observe that, over $\ZFC$, the  principle \anf{\emph{$\Ord$ is essentially subtle}}  cannot be derived from a single consistent sentence. 
In Section \ref{section:CnStronglyUnfoldable}, we will show that this observation directly follows from the results of {\cite[Section 9]{Patterns}} that show that for every natural number $n$, the assumption that $\Ord$ is essentially subtle implies the existence of class-many strongly inaccessible cardinals $\kappa$ with the property that $\VV_\kappa$ is a $\Sigma_n$-elementary submodel of $\VV$.

 \begin{proposition}\label{proposition:OrdSubtleNotFinite}
  If $\phi$ is a sentence in the language of set theory with the property that $\ZFC+\phi$ is consistent, then $${\ZFC\mathbin{+}\phi}\mathbin{\not\vdash}\anf{\textit{$\Ord$ is essentially subtle}}.$$ 
 \end{proposition}

 In contrast, assuming the consistency of sufficiently strong large cardinal assumptions, it is easy to find consistent sentences that provably imply that $\Ord$ is essentially faint. 
 As observed above, the statement that there is a proper class of subtle cardinals is an example of  a sentence with this property. 
 Our further analysis will yield additional examples of such sentences. In particular, we will produce two incompatible sentences with the given property (see Theorem \ref{MAIN:Equicons} below). 
 The following theorem now shows that the existence of such a consistent sentence is actually equivalent to the possibility of separating the two principles. Moreover, it shows that Theorem \ref{theorem:SubtleInHOD} already provides the correct consistency strength for the inequality of these principles.

\begin{theorem}\label{MAIN:SeparatePrinciples}
 The following statements are equivalent:
    \begin{enumerate}
          \item\label{item:Separate1} ${\ZFC\mathbin{+}\anf{\textit{$\Ord$ is essentially faint}}}\mathbin{\not\vdash}\anf{\textit{$\Ord$ is essentially subtle}}$. 
          
     \item\label{item:Separate2} There exists a sentence $\phi$ in the language of set theory such that the theory $\ZFC+\phi$ is consistent and $${\ZFC\mathbin{+}\phi}\mathbin{\vdash}\anf{\textit{$\Ord$ is essentially faint}}.$$

          \item\label{item:Separate3} The theory $$\ZFC+\anf{\textit{There is a proper class of subtle cardinals}}$$ is consistent. 
    \end{enumerate}
\end{theorem}

Next, we consider the possibility of axiomatizing the validity of the given principles by single sentences. Proposition \ref{proposition:OrdSubtleNotFinite} already shows that this is not possible for the principle stating that $\Ord$ is essentially subtle. A short argument in Section \ref{section:Separate} will prove the following result that provides the same conclusion for the principle \anf{\emph{$\Ord$ is essentially faint}}.

\begin{corollary}\label{corollary:AxiomatizeEssFaint}
   If the theory $$\ZFC+\anf{\textit{$\Ord$ is essentially faint}}$$ is consistent, then no sentence $\phi$ in the language of set theory  satisfies the following statements: 
 \begin{enumerate}
 
  \item\label{item:AxiomatizeEssFaint1} ${\ZFC+{\anf{\textit{$\Ord$ is essentially faint}}}}\vdash\phi$. 
  
    \item\label{item:AxiomatizeEssFaint2} ${\ZFC+\phi}\vdash{\anf{\textit{$\Ord$ is essentially faint}}}$. 
 \end{enumerate}
\end{corollary}

 We now consider the question whether the validity of the scheme {\anf{\emph{$\Ord$ is essentially subtle}}} can be finitely axiomatized over the theory $\ZFC+\anf{\textit{$\Ord$ is essentially faint}}$. 
 The following observation, proven in  Section \ref{section:Separate}, motivates the   formulation of the subsequent theorem:

 \begin{proposition}\label{proposition:Difference}
  There exists a theory $\mathrm{T}$ in the language of set theory such that the following statements hold: 
    \begin{enumerate}
    \item ${\ZFC\mathbin{+}{\anf{\textit{$\Ord$ is essentially subtle}}}}\mathbin{\vdash}\mathrm{T}$. 
        
    \item ${\ZFC\mathbin{+}{\anf{\textit{$\Ord$ is essentially faint}}}\mathbin{+}\mathrm{T}}\mathbin{\vdash}\anf{\textit{$\Ord$ is essentially subtle}}$. 
    
        \item For every sentence $\phi$ in $\mathrm{T}$, we have $${\ZFC\mathbin{+}\neg\phi}\mathbin{\vdash}\anf{\textit{$\Ord$ is essentially faint}}.$$
   \end{enumerate}
 \end{proposition}

 It is now natural to ask whether the theory $\mathrm{T}$ appearing in the above proposition can be replaced by a single sentence. The next result, also proven in Section \ref{section:Separate}, shows that the answer to this question depends on the consistency of the class-version of the large cardinal property of being a subtle limit of subtle cardinals:

 \begin{theorem}\label{theorem:DiffSentence}
   The following statements are equivalent:
   \begin{enumerate}
    \item\label{item:Diff1} There is no sentence $\phi$ in the language of set theory satisfying the following statements: 
     \begin{enumerate}
     \item\label{item:Diff1.1} ${\ZFC\mathbin{+}{\anf{\textit{$\Ord$ is essentially subtle}}}}\mathbin{\vdash}\phi$. 
        
    \item\label{item:Diff1.2} ${\ZFC\mathbin{+}{\anf{\textit{$\Ord$ is essentially faint}}}\mathbin{+}\phi}\mathbin{\vdash}\anf{\textit{$\Ord$ is essentially subtle}}$. 
    
        \item\label{item:Diff1.3} ${\ZFC\mathbin{+}\neg\phi}\mathbin{\vdash}\anf{\textit{$\Ord$ is essentially faint}}$. 
     \end{enumerate}
    
    \item\label{item:Diff2} There is no sentence $\phi$ in the language of set theory satisfying the above statements \eqref{item:Diff1.1} and \eqref{item:Diff1.2}.
    
    \item\label{item:Diff3} The theory $$\ZFC+\anf{\textit{There is a proper class of subtle cardinals}}+\anf{\textit{$\Ord$ is essentially subtle}}$$ is consistent. 
   \end{enumerate}
 \end{theorem}

 The final topic that will be treated in this paper is the relationship between the existence of weak compactness cardinals for abstract logics and the existence of strongly inaccessible cardinals. 
 By combining Theorem \ref{theorem:Makowsky} with the results of {\cite[Section 4]{zbMATH06029795}}, it is easy to see that the existence of strong compactness cardinals for all abstract logics implies the existence of a proper class of strongly inaccessible cardinals. Moreover, a combination of Theorem \ref{MAIN:EssentiallySubtle} with results in \cite[Section 9]{Patterns} (see Theorem \ref{theorem:OrdSubtle1} below) shows that the same conclusion follows from the assumption that every abstract logic has a stationary class of weak compactness cardinals. 
 In contrast, it turns out that the existence of weak compactness cardinals for abstract logics is compatible with the non-existence of strongly inaccessible cardinals. The corresponding arguments will also allow us to show that there can be incompatible sentences witnessing \eqref{item:Separate2} in Theorem \ref{MAIN:SeparatePrinciples}.

 \begin{theorem}\label{MAIN:Equicons}
  The following statements are equivalent:
     \begin{enumerate}
      \item\label{item:EquiconsNonInacc1} The theory $$\ZFC ~ + ~  \anf{\textit{There is a proper class of subtle cardinals}}$$ is consistent. 

      \item\label{item:EquiconsNonInacc2} The theory $$\ZFC + \anf{\textit{$\Ord$ is essentially faint}} + \anf{\textit{There are no strongly inaccessible cardinals}}$$ is consistent. 
      
      \item\label{item:EquiconsNonInacc3} There is a sentence $\phi$ in the language of set theory such that  $\ZFC+\phi$ is consistent and $${\ZFC\mathbin{+}\phi}\mathbin{\vdash}{\anf{\textit{$\Ord$ is essentially faint}}\mathbin{+}\anf{\textit{There are no strongly inaccessible cardinals}}}.$$ 
     \end{enumerate} 
   \end{theorem}

 While Theorem \ref{MAIN:Equicons} shows that the existence of weak compactness cardinals for all abstract logics does not imply the existence of strongly inaccessible cardinals, our analysis allows us to conclude that the given assumption at least implies the existence of strongly inaccessible cardinals in the inner model $\HOD$:

\begin{theorem}\label{theorem:LargeInHODintro}
 If $\Ord$ is essentially faint, then unboundedly many ordinals are strongly inaccessible cardinals in $\HOD$. 
\end{theorem}

 Note that the proof of this result will actually show that strongly inaccessible cardinals with certain strong combinatorial properties exist in $\HOD$ (see Corollary \ref{corollary:LargeInHOD} below).


\section{$C^{(n)}$-strongly unfoldable cardinals}\label{section:CnStronglyUnfoldable}

In this section, we review the large cardinal characterization of the principle \anf{\emph{$\Ord$ is essentially subtle}} derived in {\cite[Section 9]{Patterns}}. 
 The starting point of these results is the notion of \emph{strongly unfoldable cardinal} introduced by  Villaveces in \cite{MR1649079}. Villaveces defined an inaccessible cardinal $\kappa$ to be \emph{strongly unfoldable} if for every ordinal $\lambda$ and every transitive $\ZF^-$-model $M$ of cardinality $\kappa$ with $\kappa\in M$ and ${}^{{<}\kappa}M\subseteq M$, there is a transitive set $N$ with $\VV_\lambda\subseteq N$ and an elementary embedding $\map{j}{M}{N}$ with $\crit{j}=\kappa$ and $j(\kappa)\geq\lambda$. 
In  \cite{Patterns}, Bagaria and the author introduced natural strengthenings of this definition that demand the existence of embeddings into models with stronger correctness properties. 
 Following \cite{zbMATH06029795}, for every natural number $n$, we let $C^{(n)}$ denote the class of all \emph{$\Sigma_n$-correct} ordinals, {i.e.,} the class of all ordinals $\alpha$ with the property that $\VV_\alpha$ is a $\Sigma_n$-elementary submodel of the set-theoretic universe $\VV$, denoted by $\VV_\alpha\prec_{\Sigma_n}\VV$. 
 It is then easy to show that for each natural number $n$, the class  $C^{(n)}$ is a closed  unbounded class of ordinals that is definable by a $\Pi_n$-formula without parameters. Moreover, short arguments show that $C^{(0)}$ is the class of all ordinals and $C^{(1)}$ is the class of all cardinals $\kappa$ satisfying $\HH{\kappa}=\VV_\kappa$.  In particular, there is a sentence $\varphi$ in the language of set theory with the property that $C^{(1)}$ is equal to the class of all ordinals $\alpha$ with the property that $\varphi$ holds in $\VV_\alpha$.

\begin{definition}[\cite{Patterns}]
 Given a natural number $n$, an inaccessible cardinal $\kappa$ is \emph{$C^{(n)}$-strongly unfoldable} if for every ordinal $\lambda\in C^{(n)}$ greater than $\kappa$ and every transitive $\mathrm{ZF}^-$-model $M$ of cardinality $\kappa$ with $\kappa\in M$ and ${}^{{<}\kappa}M\subseteq M$, there is a transitive set $N$ with $\VV_\lambda\subseteq N$ and an elementary embedding $\map{j}{M}{N}$ with $\crit{j}=\kappa$,  $j(\kappa)>\lambda$ and $\VV_\lambda\prec_{\Sigma_n}\VV_{j(\kappa)}^N$. 
\end{definition}

The following basic properties of these large cardinal notions were established in {\cite[Section 9]{Patterns}}:

\begin{proposition}[\cite{Patterns}]\label{proposition:BasicCnSunf}
 \begin{enumerate}
     \item A cardinal is strongly unfoldable if and only if it is $C^{(n)}$-strongly unfoldable for some natural number $n<2$. 
     
     \item Given natural numbers $m<n$, every $C^{(n)}$-strongly unfoldable cardinal is $C^{(m)}$-strongly unfoldable. 

     \item\label{item:BasicCnSunf3} For every natural number $n$, all $C^{(n)}$-strongly unfoldable cardinals are elements of $C^{(n+1)}$. 
 \end{enumerate}
\end{proposition}

In  \cite{MR1369172}, Rathjen defined a cardinal  $\kappa$ to be  \emph{shrewd} if for every formula $\Phi(v_0,v_1)$ in the language  of set theory, every ordinal $\gamma>\kappa$ and every  $A\subseteq\VV_\kappa$ with the property that $\Phi(A,\kappa)$ holds in $\VV_\gamma$, there exist ordinals $\alpha<\beta<\kappa$ such that $\Phi(A\cap\VV_\alpha,\alpha)$ holds in $\VV_\beta$. The results of \cite{luecke2021strong} then show  that  the notions of shrewdness and strong unfoldability coincide. The proof of this equivalence relies on an embedding  characterization of shrewdness in \cite{SRminus} that closely resembles Magidor's classical characterization of supercompactness in \cite{MR295904}. The following result establishes analogous equivalences for the notion of $C^{(n)}$-strong unfoldability (see {\cite[Theorem 9.4]{Patterns}}):

\begin{lemma}[\cite{Patterns}]\label{lemma:SigmaNsunf}
 Given a natural number $n>0$, the following statements are equivalent for every cardinal $\kappa$: 
 \begin{enumerate}
     \item\label{item:sunf} The cardinal $\kappa$ is $C^{(n)}$-strongly unfoldable. 
     
     \item\label{item:shrewd} For every $\calL_\in$-formula $\varphi(v_0,v_1)$, every ordinal $\gamma\in C^{(n)}$ greater than $\kappa$, and every subset $A$ of $V_\kappa$ with the property that $\varphi(\kappa,A)$ holds in $V_\gamma$, there exist ordinals $\alpha<\beta<\kappa$ with $\beta\in C^{(n)}$ and the property that $\varphi(\alpha,A\cap V_\alpha)$ holds in $V_\beta$. 
     
    \item\label{item:Magidor} For every ordinal $\gamma\in C^{(n)}$ greater than $\kappa$ and every $z\in V_\gamma$ there exists an ordinal $\bar{\gamma}\in C^{(n)}\cap\kappa$, a cardinal $\bar{\kappa}<\bar{\gamma}$, an elementary submodel $X$ of $V_{\bar{\gamma}}$ with $V_{\bar{\kappa}}\cup\{\bar{\kappa}\}\subseteq X$, and an elementary embedding $\map{j}{X}{V_\gamma}$ with $j\restriction\bar{\kappa}=\id_{\bar{\kappa}}$, $j(\bar{\kappa})=\kappa$ and $z\in\ran{j}$.  
 \end{enumerate}
\end{lemma}

Analogous to the characterization of Vop\v{e}nka's Principle through the existence of $C^{(n)}$-extendible cardinals provided by the results of {\cite[Section 4]{zbMATH06029795}}, it  turns out that the existence $C^{(n)}$-strongly unfoldable cardinals for all natural numbers $n$ characterizes the validity of the principle \anf{\emph{$\Ord$ is essentially subtle}} (see {\cite[Theorem 1.17]{Patterns}}).

\begin{theorem}[\cite{Patterns}]\label{theorem:OrdSubtle1}
 The following schemes of axioms are equivalent over $\ZFC$: 
 \begin{enumerate}
  \item $\Ord$ is essentially subtle. 

  \item For every natural number $n$, there exists a $C^{(n)}$-strongly unfoldable cardinal. 
  
  \item For every natural number $n$, there exists a proper class of $C^{(n)}$-strongly unfoldable cardinals. 
    
 \end{enumerate}
\end{theorem}

The above results now allow us to show that no consistent sentence can imply that $\Ord$ is essentially subtle. Note that an almost identical argument shows that  Vop\v{e}nka's Principle is not the consequence of a single consistent sentence.

\begin{proof}[Proof of Proposition \ref{proposition:OrdSubtleNotFinite}]
 Assume, towards a contradiction, that there  is a sentence $\phi$ in the language of set theory such that the theory $\ZFC+\phi$ is consistent and  proves that $\Ord$ is essentially subtle. 
 Work in a model of  $\ZFC+\phi$. By combining Theorem \ref{theorem:OrdSubtle1} with Proposition \ref{proposition:BasicCnSunf}.\eqref{item:BasicCnSunf3}, we can now find an inaccessible cardinal $\kappa$ with the property that $\phi$ holds in $\VV_\kappa$. Let $\kappa$ be minimal with this property. 
 Since $\VV_\kappa$ is a model of $\ZFC+\phi$, we can   repeat this argument in $\VV_\kappa$ and find an inaccessible cardinal $\mu<\kappa$ such that $\phi$ holds in $\VV_\mu$, contradicting the minimality of $\kappa$. 
\end{proof}


\section{$C^{(n)}$-Weakly shrewd cardinals}\label{section:CnWeaklyShrewd}

In this section, we will formulate and prove an analog of Theorem \ref{theorem:OrdSubtle1} for the principle \anf{\emph{$\Ord$ is essentially faint}}. 
 The starting point of this analysis is the notion of \emph{weak shrewdness} introduced in \cite{SRminus} as a natural weakening of shrewdness. 
 This large cardinal property  still entails many of the interesting combinatorial properties of strongly unfoldable cardinals, but it no longer implies that the given cardinal is strongly inaccessible or an element of  $C^{(2)}$. 
In \cite{SRminus}, a cardinal $\kappa$ is defined to be \emph{weakly shrewd} if for every formula $\varphi(v_0,v_1)$ in the language of set theory, every cardinal $\theta>\kappa$ and every subset $z$ of $\kappa$ with the property that $\varphi(z,\kappa)$ holds in $\HH{\theta}$, there exist cardinals $\bar{\kappa}<\bar{\theta}$ such that $\bar{\kappa}<\kappa$ and $\varphi(z\cap\bar{\kappa},\bar{\kappa})$ holds in $\HH{\bar{\theta}}$. 
The results of \cite{SRminus} show that when we modify this definition to also demand that we can find cardinals $\bar{\theta}$ with the given property  below $\kappa$ (as it is the case in the definition of shrewdness),  we obtain a strictly stronger large cardinal property. In particular, the canonical argument showing that shrewd cardinal are strongly inaccessible cannot be adapted to weakly shrewd cardinals. 

 The results of {\cite[Section 3]{SRminus}}  show that every shrewd cardinal is weakly shrewd and every weakly shrewd cardinal is weakly Mahlo. Moreover, these results show that, starting from a subtle cardinal, it is possible to construct a model of set theory that contains a weakly shrewd cardinal that is not strongly inaccessible and therefore also not shrewd. Moreover, it is shown that weakly shrewd cardinals can even exist below the cardinality of the continuum. The results of this paper will show that these consistency results make use of the correct large cardinal hypothesis (see Lemma \ref{lemma:SubtleInHOD} below). 
 Motivated by Lemma \ref{lemma:SigmaNsunf}, we strengthen the notion of weak shrewdness in the same way as the transition to $C^{(n)}$-strong unfoldability strengthens the notions of shrewdness.

\begin{definition}
  Given a natural number $n$, 
  a cardinal $\kappa$ is \emph{$C^{(n)}$-weakly shrewd} if 
   for every formula $\varphi(v_0,v_1)$ in the language of set theory, 
    every cardinal $\kappa<\theta\in C^{(n)}$ and 
     every subset $z$ of $\kappa$ with the property that $\varphi(z,\kappa)$ holds in $\HH{\theta}$, 
      there exist a cardinal $\bar{\theta}\in C^{(n)}$ and  
       a cardinal $\bar{\kappa}<\min(\kappa,\bar{\theta})$ 
        such that $\varphi(z\cap\bar{\kappa},\bar{\kappa})$ holds in $\HH{\bar{\theta}}$.  
    \end{definition}

 We now derive some basic properties of this newly defined notion:

\begin{proposition}\label{proposition:BasicWeaklyShrewd}
 \begin{enumerate}
  \item\label{item:Basic1} Given natural numbers $m<n$, every $C^{(n)}$-weakly shrewd cardinal is $C^{(m)}$-weakly shrewd. 

   \item\label{item:Basic2} A cardinal is weakly shrewd if and only if it is $C^{(n)}$-weakly shrewd for some $n<2$. 
 \end{enumerate}
\end{proposition}

\begin{proof}
 \eqref{item:Basic1} Let $\kappa$ be a  $C^{(n)}$-weakly shrewd cardinal. Fix a formula $\varphi(v_0,v_1)$ in the language of set theory, a  cardinal $\kappa<\theta\in C^{(m)}$ and  $z\subseteq\kappa$ such that $\varphi(z,\kappa)$ holds in $\HH{\theta}$. Pick a cardinal $\theta<\vartheta\in C^{(n)}$. Then, in $\HH{\vartheta}$, there exists a cardinal $\kappa<\nu\in C^{(m)}$ with the property that $\varphi(z,\kappa)$ holds in $\HH{\nu}$. By our assumptions on $\kappa$, we can now find a cardinal $\bar{\vartheta}\in C^{(n)}$ and a cardinal $\bar{\kappa}<\min(\kappa,\bar{\vartheta})$ such that, in $\HH{\bar{\vartheta}}$, there exists a cardinal $\bar{\kappa}<\bar{\nu}\in C^{(m)}$ with the property that $\varphi(z\cap\bar{\kappa},\bar{\kappa})$ holds in $\HH{\bar{\nu}}$. Since $\bar{\vartheta}\in C^{(n)}$ and $n>0$, we then know that, in $\VV$, the ordinal $\bar{\nu}$ is  a cardinal in $C^{(m)}$ with the property that $\varphi(z\cap\bar{\kappa},\bar{\kappa})$ holds in $\HH{\bar{\nu}}$. 
 
  \eqref{item:Basic2} By definition and  \eqref{item:Basic1}, it suffices to show that every weakly shrewd cardinal is $C^{(1)}$-weakly shrewd. Let $\kappa$ be a weakly shrewd cardinal, let $\varphi(v_0,v_1)$ be a formula in the language of set theory, let $\kappa<\theta\in C^{(1)}$ be a cardinal and let $z$ be a subset of $\kappa$ with the property that $\varphi(z,\kappa)$ holds in $\HH{\theta}$. Since $\theta$ is an element of $C^{(1)}$, we then know that, in $\HH{\theta}$, for every ordinal $\alpha$, the class $\VV_\alpha$ is a set. Our assumption then allows us to find a cardinal $\bar{\theta}$ and a cardinal $\bar{\kappa}<\min(\kappa,\bar{\theta})$ with the property that, in $\HH{\bar{\theta}}$, the statement $\varphi(z\cap\bar{\kappa},\bar{\kappa})$ holds and for every ordinal $\alpha$, the class $\VV_\alpha$ is a set. This completes the proof, because this conclusion ensures that   $\bar{\theta}$ is an element of $C^{(1)}$.  
\end{proof}

 In combination with {\cite[Proposition 3.3]{SRminus}}, the above proposition shows that all $C^{(n)}$-weakly shrewd cardinals are weakly Mahlo. In particular, if a $C^{(n)}$-weakly shrewd cardinal is an element of $C^{(1)}$, then it is strongly  inaccessible. 
Motivated by {\cite[Lemma 3.1]{SRminus}}, we now prove an analog of Lemma \ref{lemma:SigmaNsunf} for $C^{(n)}$-weak shrewdness.

\begin{lemma}\label{lemma:WcnShrewdChar}
  Given a natural number $n>0$, the following statements are equivalent for every cardinal $\kappa$: 
  \begin{enumerate}
   \item\label{item:WcnShrewdChar1} The cardinal $\kappa$ is $C^{(n)}$-weakly  shrewd. 
   
   \item\label{item:WcnShrewdChar2} For every cardinal $\kappa<\theta\in C^{(n)}$ and every $y\in \HH{\theta}$, there exists a cardinal $\bar{\theta}\in C^{(n)}$, a cardinal $\bar{\kappa}<\min(\kappa,\bar{\theta})$, an elementary submodel $X$ of $\HH{\bar{\theta}}$ with $\bar{\kappa}+1\subseteq X$ and an elementary embedding $\map{j}{X}{\HH{\theta}}$ with $j\restriction\bar{\kappa}=\id_{\bar{\kappa}}$, $j(\bar{\kappa})=\kappa$ and $y\in\ran{j}$. 
  \end{enumerate}
\end{lemma}

\begin{proof}
 First, assume that \eqref{item:WcnShrewdChar1} holds, $\kappa<\theta\in C^{(n)}$ is a cardinal and $y$ is an element of $\HH{\theta}$. Pick a  cardinal $\theta<\vartheta\in C^{(n)}$ and an elementary submodel $Y$ of $\HH{\theta}$ of cardinality $\kappa$ with the property that $\kappa\cup\{\kappa,y\}\subseteq Y$. In addition, fix a bijection $\map{b}{\kappa}{Y}$ satisfying  $b(0)=\kappa$, $b(1)=\langle \kappa,y\rangle$ and $b(\omega\cdot(1+\alpha))=\alpha$ for all $\alpha<\kappa$. Finally, define $z$ to be the set of all elements of $\kappa$ of the form $\goedel{\ell}{\alpha_0,\ldots,\alpha_{n-1}}$\footnote{We use  $\goedel{\cdot}{\ldots,\cdot}$ to denote (iterated applications of)  the \emph{G\"odel pairing function}.}
 with the property that $\ell<\omega$ is the G\"odel number of a set-theoretic formula with $n$ free variables, $\alpha_0,\ldots,\alpha_{n-1}<\kappa$ and $\mathsf{Sat}(Y,\ell,\langle b(\alpha_0),\ldots,b(\alpha_{n-1})\rangle)$ holds, where $\mathsf{Sat}$ denotes the formalized satisfaction relation (see, for example, {\cite[Section I.9]{MR750828}}). 
 
 By our assumption and the fact that $\mathsf{Sat}$ is uniformly $\Delta_1$-definable over models of $\ZF^-$, we can now find a cardinal $\bar{\vartheta}\in C^{(n)}$ and a cardinal $\bar{\kappa}<\min(\kappa,\bar{\vartheta})$ with the property that, in $\HH{\bar{\vartheta}}$, there exists a cardinal $\bar{\kappa}<\bar{\theta}\in C^{(n)}$, an elementary submodel $X$ of $\HH{\bar{\theta}}$ of cardinality $\bar{\kappa}$ with $\bar{\kappa}+1\subseteq X$ and a bijection $\map{a}{\bar{\kappa}}{X}$ with $a(0)=\bar{\kappa}$,  $a(\omega\cdot(1+\alpha))=\alpha$ for all $\alpha<\bar{\kappa}$ and the property that the set $z\cap\bar{\kappa}$ consists of all elements of $\kappa$ of the form $\goedel{\ell}{\alpha_0,\ldots,\alpha_{n-1}}$ such that $\ell<\omega$ is the G\"odel number of a set-theoretic formula with $n$ free variables, $\alpha_0,\ldots,\alpha_{n-1}<\bar{\kappa}$ and $\mathsf{Sat}(X,\ell,\langle a(\alpha_0),\ldots,a(\alpha_{n-1})\rangle)$ holds. 
  Since $\bar{\vartheta}\in C^{(n)}$, we  now know that, in $\VV$, the ordinal $\bar{\theta}$ is a cardinal in $C^{(n)}$ and the set $X$ is an elementary submodel of $\HH{\bar{\theta}}$. If we now define $$\map{j ~ = ~ b\circ a^{{-}1}}{X}{\HH{\theta}},$$ then it is easy to check that $j$ is an elementary embedding with $j\restriction\bar{\kappa}=\id_{\bar{\kappa}}$, $j(\bar{\kappa})=\kappa$ and $y\in\ran{j}$. 
 
 Now, assume that \eqref{item:WcnShrewdChar2} holds, $\varphi(v_0,v_1)$ is a formula in the language of set theory, $\kappa<\theta\in C^{(n)}$ is a cardinal and $z$ is a subset of $\kappa$ with the property that $\varphi(z,\kappa)$ holds in $\HH{\theta}$. Our assumptions then allow us to find a cardinal $\bar{\theta}\in C^{(n)}$, a cardinal $\bar{\kappa}<\min(\kappa,\bar{\theta})$, an elementary submodel $X$ of $\HH{\bar{\theta}}$ with $\bar{\kappa}+1\subseteq X$ and an elementary embedding $\map{j}{X}{\HH{\theta}}$ with $j\restriction\bar{\kappa}=\id_{\bar{\kappa}}$, $j(\bar{\kappa})=\kappa$ and $z\in\ran{j}$. We then know that $z\cap\bar{\kappa}\in X$ and $j(z\cap\bar{\kappa})=z$. But, this shows that $\varphi(z\cap\bar{\kappa},\bar{\kappa})$ holds in both $X$ and $\HH{\bar{\theta}}$.  
\end{proof}

Next, we observe that the existence of subtle cardinals entails the existence of many $C^{(n)}$-weakly shrewd cardinals:

\begin{lemma}\label{lemma:ShrewdFromSubtle}
 Given a natural number $n>0$, every subtle cardinal is a stationary limit of strongly inaccessible $C^{(n)}$-weakly shrewd cardinals. 
\end{lemma}

\begin{proof}
 Assume, towards a contradiction, that there is a subtle cardinal $\delta$ and a closed unbounded subset $C$ of $\delta$ that consists of uncountable  cardinals that are not $C^{(n)}$-weakly shrewd. Given $\gamma\in C$, we can use Lemma \ref{lemma:WcnShrewdChar} to find $\ell_\gamma<\omega$ and $z_\gamma\subseteq\gamma$ such that $\ell_\gamma$ is the G\"odel number of a  set-theoretic formula $\varphi(v_0,v_1)$ with the property that there exists a cardinal $\gamma<\theta\in C^{(n)}$ such  that $\varphi(z,\gamma)$ holds in $\HH{\theta}$  and  for every cardinal $\bar{\theta}\in C^{(n)}$, there is no cardinal $\bar{\gamma}<\min(\gamma,\bar{\theta})$ with the property that $\varphi(z\cap\bar{\gamma},\bar{\gamma})$ holds in $\HH{\bar{\theta}}$. Now, let $\seq{E_\gamma}{\gamma\in C}$ denote the unique sequence with $$E_\gamma ~ = ~ \goedel{0}{\ell_\gamma} ~ \cup ~ \Set{\goedel{1}{\beta}}{\beta\in z_\gamma}$$ for all  $\gamma\in C$. Using the subtlety of $\delta$, we can now find $\beta<\gamma$ in $C$ with $E_\gamma\cap\beta=E_\beta$. We then know that $\ell_\beta=\ell_\gamma$ and $z_\gamma\cap\beta=z_\beta$. But, this contradicts the choice of $\ell_\gamma$ and $z_\gamma$, because, if $\varphi(v_0,v_1)$ is the formula coded by $\ell_\gamma$, then $\beta$ is a cardinal below $\gamma$ with the property that there exists $\beta<\theta\in C^{(n)}$ such that $\varphi(z_\gamma\cap\beta,\beta)$ holds in $\HH{\theta}$.  
 
 Since subtle cardinals are regular limit points of the class $C^{(1)}$, the above computations show that subtle cardinals are stationary limits of  $C^{(n)}$-weakly shrewd cardinals that are elements of $C^{(1)}$. This yields the statement of the lemma, because, as observed earlier, $C^{(n)}$-weakly shrewd cardinals are inaccessible if and only if they are elements of $C^{(1)}$.  
\end{proof}

We close this section by proving the desired characterization of the   principle \anf{\emph{$\Ord$ is essentially faint}} through the existence of $C^{(n)}$-weakly shrewd cardinals:

\begin{theorem}\label{theorem:CharOrdFaint}
 The following schemes of axioms are equivalent over $\ZFC$: 
 \begin{enumerate}
  \item\label{item:CharOrdFaint1} $\Ord$ is essentially faint. 
  
  \item\label{item:CharOrdFaint2} For every natural number $n$, there exists a proper class of $C^{(n)}$-weakly shrewd cardinals. 
 \end{enumerate}
\end{theorem}

\begin{proof}
 First, assume, towards a contradiction, that \eqref{item:CharOrdFaint1} holds, $n>0$ is a  natural number  and there exists an ordinal $\xi$ with the property that there are no $C^{(n)}$-weakly shrewd cardinals above $\xi$. Then there is a unique class function $E$ with domain $\Ord$ such that the following statements hold: 
  \begin{itemize}
   \item $E(\alpha)=\POT{\alpha}$ for all ordinal $\alpha\leq\xi$. 
   
   \item $E(\alpha+1)=\{\{\goedel{0}{\alpha}\}\}$ for all ordinals $\alpha\geq\xi$.  
   
   \item If $\alpha>\xi$ is a singular limit ordinal, then $E(\alpha)$ consists of all subsets of $\alpha$ of the form $$\{\goedel{1}{\cof{\alpha}}\} ~ \cup ~ \Set{\goedel{2}{\eta,c(\eta)}}{\eta<\cof{\alpha}},$$ where $\map{c}{\cof{\alpha}}{\alpha}$ is a strictly increasing cofinal function. 
   
   \item If $\alpha>\xi$ is a regular cardinal, then $E(\alpha)$ consists of all subsets of $\alpha$ of the form $$\{\goedel{3}{\ell}\} ~ \cup ~ \Set{\goedel{4}{\eta}}{\eta\in z},$$ where $\ell$ is the G\"odel number of a set-theoretic formula $\varphi(v_0,v_1)$ and $z$ is a subset of $\alpha$ with the property that there exists a cardinal $\alpha<\theta\in C^{(n)}$ such that $\varphi(z,\alpha)$ holds in $\HH{\theta}$ and for every cardinal $\bar{\theta}\in C^{(n)}$, there is no cardinal $\bar{\alpha}<\min(\alpha,\bar{\theta})$ with the property that $\varphi(z\cap\bar{\alpha},\bar{\alpha})$ holds in $\HH{\bar{\theta}}$. 
  \end{itemize}
 
 Our assumption then ensures that $E(\alpha)\neq\emptyset$ holds for all $\alpha\in\Ord$. Therefore, we can find ordinals $\xi<\alpha<\beta$ and $A\in E(\beta)$ with $A\cap\alpha\in E(\alpha)$. The definition of $E$ then ensures that $\alpha$ and $\beta$ are both regular cardinals. Fix a set-theoretic formula $\varphi(v_0,v_1)$  and a subset $z$ of $\beta$ such that $A=\{\goedel{3}{\ell}\}\cup\Set{\goedel{4}{\eta}}{\eta\in z}$, where $\ell$ is the G\"odel number of $\varphi(v_0,v_1)$. Since $$A\cap\alpha ~ = ~ \{\goedel{3}{\ell}\}\cup\Set{\goedel{4}{\eta}}{\eta\in z\cap\alpha} ~ \in ~ E(\alpha),$$ we now know that there exists a cardinal $\alpha<\theta\in C^{(n)}$ with the property that $\varphi(z\cap\alpha,\alpha)$ holds in $\HH{\theta}$. Since $\alpha<\min(\beta,\theta)$, this contradicts the fact that $A$ is an element of $E(\beta)$. 
  
 Now, assume that \eqref{item:CharOrdFaint2} holds and $E$ is a class function with domain $\Ord$ and the property that $\emptyset\neq E(\alpha)\subseteq\POT{\alpha}$ holds for all $\alpha\in\Ord$. Fix an ordinal $\xi$. Pick a natural number $n>0$ such that there exists a $\Sigma_n$-formula $\varphi(v_0,v_1,v_2)$ and a parameter $y$  defining $E$. By our assumption, there is a $C^{(n)}$-weakly shrewd cardinal $\kappa>\xi$ with $y\in\HH{\kappa}$. Pick $A\in E(\kappa)$ and a cardinal $\kappa<\theta\in C^{(n)}$. Then Lemma \ref{lemma:WcnShrewdChar} yields a cardinal $\bar{\theta}\in C^{(n)}$, a cardinal $\bar{\kappa}<\min(\kappa,\bar{\theta})$, an elementary submodel $X$ of $\HH{\bar{\theta}}$ with $\bar{\kappa}+1\subseteq X$ and an elementary embedding $\map{j}{X}{\HH{\theta}}$ with $j\restriction\bar{\kappa}=\id_{\bar{\kappa}}$, $j(\bar{\kappa})=\kappa$ and $A,y,\xi\in\ran{j}$. We then know that $\xi<\bar{\kappa}$, $y\in\HH{\bar{\kappa}}\cap X$ with $j(y)=y$ and $A\cap\bar{\kappa}\in X$ with $j(A\cap\bar{\kappa})=A$. Moreover, the fact that $\theta\in C^{(n)}$ ensures that, in $\HH{\theta}$, there exists a set $\calA$ such that $A\in\calA$ and $\varphi(\kappa,\calA,y)$ holds. This implies that, in $\HH{\bar{\theta}}$,  there exists a set $\bar{\calA}$ such that $A\cap\bar{\kappa}\in\bar{\calA}$ and $\varphi(\bar{\kappa},\bar{\calA},y)$ holds. Since $\bar{\theta}\in C^{(n)}$, we then know that $\varphi(\bar{\kappa},\bar{\calA},y)$ also holds in $\VV$ and we can conclude that  $A\cap\bar{\kappa}\in\bar{\calA}=E(\bar{\kappa})$. 
\end{proof}


\section{Separating the schemes}\label{section:Separate}

The goal of this section is to identify scenarios in which the principle \anf{\emph{$\Ord$ is essentially faint}} holds while the principle \anf{\emph{$\Ord$ is essentially subtle}} fails. In the light of Theorems \ref{theorem:OrdSubtle1} and \ref{theorem:CharOrdFaint}, this task is closely connected to the analysis of $C^{(n)}$-weakly shrewd cardinals that are not $C^{(n)}$-strongly unfoldable. For the case $n=1$, this analysis is provided by   {\cite[Lemma 3.5]{SRminus}}, which shows that a weakly shrewd cardinal is strongly unfoldable if and only if it is an element of $C^{(2)}$. The following lemma provides a direct analog of this result  for $C^{(n)}$-weakly shrewd cardinals:

\begin{lemma}\label{lemma:separatePropertiesCn}
Given a natural number $n>0$, the following statements are equivalent for every  $C^{(n)}$-weakly shrewd cardinal: 
 \begin{enumerate}
  \item\label{item:WCnShredNonSunf1} The cardinal $\kappa$ is $C^{(n)}$-strongly unfoldable. 
  
  \item\label{item:WCnShredNonSunf2} The cardinal $\kappa$ is an element of $C^{(n+1)}$. 
  
  \item\label{item:WCnShredNonSunf3} There is no ordinal $\varepsilon\geq\kappa$ with the property that the set $\{\varepsilon\}$ is definable by a $\Sigma_{n+1}$-formula with parameters in $\HH{\kappa}$. 
 \end{enumerate}
\end{lemma}

\begin{proof}
 First, assume that \eqref{item:WCnShredNonSunf1} fails and \eqref{item:WCnShredNonSunf3} holds. 
  Given $\alpha<\kappa$, the fact that the set $\{\betrag{\VV_\alpha}\}$ is definable by a $\Sigma_2$-formula with parameter $\alpha$ allows us to use our assumption to conclude that $\betrag{\VV_\alpha}<\kappa$ holds. 
  In particular, we know that the set $\VV_\kappa$ has cardinality $\kappa$ in this setting. 
  Since \eqref{item:WCnShredNonSunf1} fails, Lemma \ref{lemma:SigmaNsunf} shows that there is a set-theoretic formula $\varphi(v_0,v_1)$, a cardinal $\kappa<\theta\in C^{(n)}$ and a subset $A$ of $\VV_\kappa$  such that $\varphi(A,\kappa)$ holds in $\HH{\theta}$ and there is no cardinal $\bar{\theta}\in C^{(n)}\cap\kappa$ such that $\varphi(A\cap\VV_{\bar{\kappa}},\bar{\kappa})$ holds in $\HH{\bar{\theta}}$ for some cardinal $\bar{\kappa}<\bar{\theta}$. 
  Using the fact that $\kappa$ is weakly inaccessible and the set $\VV_\kappa$ has cardinality $\kappa$, we can now find a set-theoretic formula $\psi(v_0,v_1)$ and a subset $z$ of $\kappa$ with the property that $\psi(z,\kappa)$ holds in $\HH{\theta}$ and there is no cardinal $\bar{\theta}\in C^{(n)}\cap\kappa$ such that $\psi(z\cap\bar{\kappa},\bar{\kappa})$ holds in $\HH{\bar{\theta}}$ for some cardinal $\bar{\kappa}<\bar{\theta}$. 
  Since $\kappa$ is $C^{(n)}$-weakly shrewd, we can find a cardinal $\bar{\theta}\in C^{(n)}$ and a cardinal $\bar{\kappa}<\min(\kappa,\bar{\theta})$ such that $\varphi(z\cap\bar{\kappa},\bar{\kappa})$ holds in $\HH{\bar{\theta}}$. 
  Now, define $\varepsilon$ to be the least element of $C^{(n)}$ above $\bar{\kappa}$ with the property that $\varphi(z\cap\bar{\kappa},\bar{\kappa})$ holds in $\HH{\varepsilon}$. 
  We then know that $\varepsilon\geq\kappa$ and it is easy to see that the set $\{\varepsilon\}$ is definable by a $\Sigma_{n+1}$-formula with parameters $\bar{\kappa},z\cap\bar{\kappa}\in\HH{\kappa}$. 
 This contradicts \eqref{item:WCnShredNonSunf3} and  we can conclude that \eqref{item:WCnShredNonSunf3} implies \eqref{item:WCnShredNonSunf1}. 
  
 These computations complete the proof of the lemma, because \eqref{item:WCnShredNonSunf2} obviously implies \eqref{item:WCnShredNonSunf3} and Proposition \ref{proposition:BasicCnSunf}.\eqref{item:BasicCnSunf3} shows that  \eqref{item:WCnShredNonSunf1} implies \eqref{item:WCnShredNonSunf2}.  
\end{proof}

\begin{corollary}\label{corollary:SubtleNotUnfoldable}
  Given a natural number $n>0$, every subtle cardinal that is not a limit of subtle cardinals is a stationary limit of strongly inaccessible $C^{(n)}$-weakly shrewd cardinals that are not $C^{(n)}$-strongly unfoldable. 
\end{corollary}

\begin{proof}
 Let $\delta$ be a subtle cardinal that is not a limit of subtle cardinals and let $\gamma<\delta$ be an ordinal with the property that the interval $(\gamma,\delta)$ does not contain subtle cardinals. Then the set $\{\delta\}$ can be defined by a $\Sigma_2$-formula with parameter $\gamma$. In this situation, Lemma \ref{lemma:separatePropertiesCn} shows that no $C^{(n)}$-weakly shrewd cardinal in the interval $(\gamma,\delta)$ is $C^{(n)}$-strongly unfoldable. An application of Lemma \ref{lemma:ShrewdFromSubtle} now yields the statement of the corollary. 
\end{proof}

The next lemma provides us with a tool to derive non-trivial consistency strength from the existence of   $C^{(n)}$-weakly shrewd cardinals that are not $C^{(n)}$-strongly unfoldable. 
 An interesting aspect of this result is the fact that, 
 in contrast to most results that provide lower bounds for the consistency strength of certain assumptions,  
 it establishes  the existence of large cardinals not only in certain  fine-structural inner models ({e.g.,} in  the constructible universe $\LL$) 
 but also  in the inner model $\HOD$. 
  Since the results of this paper show that the given assumption is compatible with the non-existence of strongly inaccessible cardinals, 
  the following lemma reveals that the existence of weakly shrewd cardinals that are not strongly unfoldable has interesting and non-trivial consequences for the behavior of ordinal definable sets.

\begin{lemma}\label{lemma:SubtleInHOD}
 Let $n>0$ be a natural number and let $\kappa$ be a $C^{(n)}$-weakly shrewd cardinal that is not $C^{(n)}$-strongly unfoldable. If $\delta$ is the least ordinal $\varepsilon\geq\kappa$ with the property that the set $\{\varepsilon\}$ is definable by a $\Sigma_{n+1}$-formula with parameters in $\HH{\kappa}$, then the following statements hold: 
  \begin{enumerate}
   \item $\delta$ is a  cardinal greater than $\kappa$ with $\cof{\delta}\in\kappa\cup\{\delta\}$. 
   
   \item $\delta$ is a subtle cardinal in $\HOD$. 
  \end{enumerate}
\end{lemma}

\begin{proof}
 Fix a $\Sigma_{n+1}$-formula $\varphi(v_0,v_1)$ and an element $y$ of $\HH{\kappa}$ with the property that $\delta$ is the unique set $x$ such that $\varphi(x,y)$ holds. Pick a cardinal $\kappa<\theta\in C^{(n+1)}$ and use  Lemma \ref{lemma:WcnShrewdChar} to find  a cardinal $\bar{\theta}\in C^{(n)}$, a cardinal $\bar{\kappa}<\min(\kappa,\bar{\theta})$, an elementary submodel $X$ of $\HH{\bar{\theta}}$ with $\bar{\kappa}+1\subseteq X$ and an elementary embedding $\map{j}{X}{\HH{\theta}}$ with $j\restriction\bar{\kappa}=\id_{\bar{\kappa}}$, $j(\bar{\kappa})=\kappa$ and $y,\delta\in\ran{j}$. Then $y\in\HH{\bar{\kappa}}$ with $j(y)=y$. Pick $\varepsilon\in X$  with $j(\varepsilon)=\delta$. Since $\varphi(\delta,y)$ holds in $\HH{\theta}$, we then know that $\varphi(\varepsilon,y)$ holds in $\HH{\bar{\theta}}$ and the fact that $\bar{\theta}$ is an element of $C^{(n)}$ implies that $\varphi(\varepsilon,y)$ also holds in $\VV$. This shows that $\delta=\varepsilon$ and hence $\delta>\kappa$. 
  Next, note that the set $\{\betrag{\delta}\}$ is $\Sigma_2$-definable from the parameter $\delta$ and, since $\betrag{\delta}\geq\kappa$, the minimality of $\delta$ implies that $\delta$ is a cardinal. Finally, since the set   $\{\cof{\delta}\}$ is $\Sigma_2$-definable from the parameter $\delta$, the minimality of $\delta$ also ensures that either $\delta$ is regular or $\cof{\delta}<\kappa$ holds.  
 
 Now, assume that $\delta$ is not a subtle cardinal in $\HOD$. Let $\langle\vec{A},C\rangle$ denote the least pair in the canonical well-ordering of $\HOD$ with the property that $\vec{A}=\seq{A_\gamma}{\gamma<\delta}$ is a sequence of length $\delta$ with $A_\gamma\subseteq\gamma$ for all $\gamma<\delta$ and $C$ is a closed unbounded subset of $\delta$ with the property that $A_\gamma\cap\beta\neq A_\beta$ holds for all $\beta<\gamma$ in $C$. Since the class of all proper initial segments of the canonical well-ordering of $\HOD$ is definable in $\VV$ by a $\Sigma_2$-formula without parameters  (see {\cite[Lemma 13.25]{MR1940513}}), we then know that, in $\VV$,  the sets $\{\vec{A}\}$ and $\{C\}$ are both definable by $\Sigma_2$-formulas with parameter $\delta$, and hence these sets are also definable by $\Sigma_{n+1}$-formulas with parameter $y$. We then know that $C\cap\kappa\neq\emptyset$, because otherwise $\kappa\leq\min(C)<\delta$ is an ordinal  with the property that the set $\{\min(C)\}$ is definable by a $\Sigma_{n+1}$-formula with parameters in $\HH{\kappa}$. 
 
 \begin{claim*}
  $\kappa\in C$.
 \end{claim*}
 
 \begin{proof}[Proof of the Claim]
  Assume, towards a contradiction, that $\kappa$ is not an element of $C$ and set $\alpha=\max(C\cap\kappa)<\kappa$ and $\beta=\min(C\setminus\kappa)=\min(C\setminus(\alpha+1))<\delta$. Then $\beta$ is an ordinal greater than $\kappa$ with the property that the set $\{\beta\}$  is definable by a $\Sigma_{n+1}$-formula with parameters $\alpha,y\in\HH{\kappa}$ and hence it follows that $\beta\geq\delta$, a contradiction. 
 \end{proof}
 
 Pick a cardinal $\delta<\theta\in C^{(n+1)}$ and use  Lemma \ref{lemma:WcnShrewdChar} to find 
  a cardinal $\bar{\theta}\in C^{(n)}$, a cardinal $\bar{\kappa}<\min(\kappa,\bar{\theta})$, an elementary submodel $X$ of $\HH{\bar{\theta}}$ with $\bar{\kappa}+1\subseteq X$ and an elementary embedding $\map{j}{X}{\HH{\theta}}$ with $j\restriction\bar{\kappa}=\id_{\bar{\kappa}}$, $j(\bar{\kappa})=\kappa$ and $y\in\ran{j}$. We then again know that $y\in\HH{\bar{\kappa}}$ with $j(y)=y$ and, by the correctness properties of $\HH{\theta}$ and $\HH{\bar{\theta}}$, the fact that the sets $\{\vec{A}\}$, $\{C\}$ and $\{\delta\}$ are all definable by $\Sigma_{n+1}$-formulas with parameter $y$ implies that $\vec{A},C,\delta\in X$ with $j(\vec{A})=\vec{A}$, $j(C)=C$ and $j(\delta)=\delta$. In this situation, elementarity allows us to conclude that $\bar{\kappa}\in C$ and  $A_{\bar{\kappa}}\in X$ with $j(A_{\bar{\kappa}})=A_\kappa$ and $A_\kappa\cap\bar{\kappa}=j(A_{\bar{\kappa}})\cap\bar{\kappa}=A_{\bar{\kappa}}$, contradicting the choice of $\vec{A}$  and $C$. 
\end{proof}

As a first application of the above lemma, we observe that, by combining it with Corollary \ref{corollary:SubtleNotUnfoldable}, we can  determine the exact consistency strength of the existence of a $C^{(n)}$-weakly shrewd cardinal that is not $C^{(n)}$-strongly unfoldable. The following corollary directly strengthens {\cite[Theorem 1.9.(i)]{SRminus}}. Note that, in combination with Theorem \ref{theorem:OrdSubtle1}, this result shows that, over $\ZFC$, the existence of  a $C^{(n)}$-weakly shrewd cardinal that is not $C^{(n)}$-strongly unfoldable has strictly greater consistency strength than the existence of a $C^{(n)}$-weakly shrewd cardinal.

\begin{corollary}
Given a natural number $n>0$, the following statements are equiconsistent over $\ZFC$: 
 \begin{enumerate}
  \item\label{item:consSubtle1} There exists a $C^{(n)}$-weakly shrewd cardinal that is not $C^{(n)}$-strongly unfoldable. 
  
  \item\label{item:consSubtle2} There exists a subtle cardinal. \qed 
 \end{enumerate}
\end{corollary}

Next, we use Lemma \ref{lemma:SubtleInHOD} to show that the principles \anf{\emph{$\Ord$ is essentially subtle}} and \anf{\emph{$\Ord$ is essentially faint}} are equivalent if only boundedly many  cardinals are subtle in $\HOD$.

\begin{proof}[Proof of Theorem \ref{theorem:SubtleInHOD}]
  Assume, towards a contradiction that $\Ord$ is essentially faint and not essentially subtle. 
  By Theorem \ref{theorem:OrdSubtle1}, the fact that $\Ord$ is not essentially subtle implies that for some natural number $n>0$, there are no $C^{(n)}$-strongly unfoldable cardinals. Since $\Ord$ is  essentially faint, Theorem \ref{theorem:CharOrdFaint} now shows that  there is  a proper class of $C^{(n)}$-weakly shrewd cardinals that are not $C^{(n)}$-strongly unfoldable. 
  In this situation, we can apply  Lemma \ref{lemma:SubtleInHOD} to conclude that a proper class of cardinals is subtle in $\HOD$.  
\end{proof}

 We now continue by using Theorem \ref{theorem:SubtleInHOD} to show that the principle \anf{\emph{$\Ord$ is essentially faint}}  is not finitely axiomatizable over $\ZFC$.

\begin{proof}[Proof of Corollary \ref{corollary:AxiomatizeEssFaint}]
   Assume,   towards a contradiction, that the theory $$\ZFC+\anf{\textit{$\Ord$ is essentially faint}}$$ is consistent and    there is a sentence $\phi$ in the language of set theory with the property that the statements   \eqref{item:AxiomatizeEssFaint1} and  \eqref{item:AxiomatizeEssFaint2} listed in the corollary hold. 
     Define $\psi$ to be the conjunction of $\phi$ with the statement that no ordinal is a subtle cardinal in $\HOD$.

     \begin{claim*}
      The theory $\ZFC+\psi$ is consistent. 
     \end{claim*}    
    
    \begin{proof}[Proof of the Claim]
      By our assumption and Corollary \ref{corollary:EquiconsEsubtleEfaint}, we know that the theory $$\ZFC+\anf{\textit{$\Ord$ is essentially subtle}}$$ is consistent and, since subtle cardinals yield set-sized models of this theory, it follows that the theory $$\ZFC ~ +  ~ \anf{\textit{$\Ord$ is essentially subtle}} ~ + ~ \anf{\textit{No ordinal is a subtle cardinal in $\HOD$}}$$ is also consistent. By \eqref{item:AxiomatizeEssFaint1}, this theory proves $\psi$.  
    \end{proof}

        By  \eqref{item:AxiomatizeEssFaint2}, we can  now apply  Theorem \ref{theorem:SubtleInHOD}  to conclude that $$\ZFC+\psi\vdash\anf{\textit{$\Ord$ is essentially subtle}}$$ holds. In this situation, Proposition \ref{proposition:OrdSubtleNotFinite} shows that $\ZFC+\psi$ is inconsistent, contradicting the above claim.   
\end{proof}

Another application of Theorem \ref{theorem:SubtleInHOD} yields the following strengthening of Theorem \ref{theorem:LargeInHODintro}:

\begin{corollary}\label{corollary:LargeInHOD}
 Given a natural number $n$, if $\Ord$ is essentially faint, then unboundedly many ordinals are strongly inaccessible $C^{(n)}$-weakly shrewd cardinals in $\HOD$. 
\end{corollary}

\begin{proof}
  If $\Ord$ is essentially subtle, then $\Ord$ is essentially subtle in $\HOD$ and an application of Theorem \ref{theorem:OrdSubtle1} in $\HOD$ yields the desired conclusion. Hence, we may assume that $\Ord$ is not essentially subtle. Then Theorem \ref{theorem:SubtleInHOD} shows that, in $\HOD$, there is a proper class of subtle cardinals. An application of  Lemma \ref{lemma:ShrewdFromSubtle} now allows us to conclude that unboundedly many ordinals are inaccessible $C^{(n)}$-weakly shrewd cardinals in $\HOD$.    
\end{proof}

 We now continue by using the developed techniques to show that the consistency strength of the inequality of the principles \anf{\emph{$\Ord$ is essentially subtle}} and   \anf{\emph{$\Ord$ is essentially faint}} is equal to the existence of a proper class of subtle cardinals.

\begin{proof}[Proof of Theorem \ref{MAIN:SeparatePrinciples}]
 First, assume that  $${\ZFC\mathbin{+}\anf{\textit{$\Ord$ is essentially faint}}}\mathbin{\not\vdash}\anf{\textit{$\Ord$ is essentially subtle}}$$ holds. Then  Theorem \ref{theorem:OrdSubtle1} yields a natural number $n>0$ with the property that the theory $$\ZFC+\anf{\textit{$\Ord$ is essentially faint}}+\anf{\textit{There are no $C^{(n)}$-strongly unfoldable cardinals}}$$ is consistent and we can work in a model of this theory. An application of  Theorem \ref{theorem:SubtleInHOD} now directly shows that, in $\HOD$, there is a proper class of subtle cardinals. These computations show that \eqref{item:Separate1} implies \eqref{item:Separate3} in the statement of the theorem.

 Next, let $\phi$ denote the set-theoretic sentence stating that there is a proper class of subtle cardinals. A combination of Lemma \ref{lemma:ShrewdFromSubtle} with Theorem \ref{theorem:CharOrdFaint} then shows that the theory $\ZFC+\phi$ proves that $\Ord$ is essentially faint. This shows that \eqref{item:Separate3} implies \eqref{item:Separate2} in the statement of the theorem.

 Finally, assume that there is a sentence $\phi$   such that the theory $\ZFC+\phi$ is consistent and $${\ZFC\mathbin{+}\phi}\mathbin{\vdash}\anf{\textit{$\Ord$ is essentially faint}}.$$ Then Proposition \ref{proposition:OrdSubtleNotFinite} ensures that $${\ZFC\mathbin{+}\phi}\mathbin{\not\vdash}\anf{\textit{$\Ord$ is essentially subtle}}$$ and hence we can conclude that $${\ZFC\mathbin{+}\anf{\textit{$\Ord$ is essentially faint}}}\mathbin{\not\vdash}\anf{\textit{$\Ord$ is essentially subtle}}.$$ This shows that   \eqref{item:Separate2} implies \eqref{item:Separate1} in the statement of the theorem.  
\end{proof}

 In order to expand our analysis of the given class principles, we now explore a phenomenon already unveiled by {\cite[Theorem 1.11]{SRminus}} showing that the existence of weakly shrewd cardinals that are not strongly unfoldable yields reflection properties for statements of arbitrary high complexities. 
 The following lemma strengthens the result from \cite{SRminus} by  both weakening the  assumption used and establishing the existence of large cardinals that induce the given reflection properties:

 \begin{lemma}\label{lemma:Overspil}
 Given  natural numbers $n>m>0$, if $\kappa$ is  a $C^{(m)}$-weakly shrewd cardinal that is not $C^{(m)}$-strongly unfoldable, $\delta>\kappa$ is a cardinal with the property that the set $\{\delta\}$ is definable by a $\Sigma_{m+1}$-formula with parameters in $\HH{\kappa}$ and $\alpha<\kappa$ is an ordinal, then the interval $(\alpha,\delta)$ contains a  $C^{(n)}$-weakly shrewd cardinal that is not $C^{(n)}$-strongly unfoldable. 
 \end{lemma}
 
 \begin{proof}
  Pick $y\in\HH{\kappa}$ with the property that the set $\{\delta\}$ is definable by a $\Sigma_{m+1}$-formula with parameter $y$.   Since $\kappa$ is a limit cardinal, we can find a  cardinal $\alpha<\rho<\kappa$ with $y\in\HH{\rho}$.  Moreover, since $n>m$, it follows that  the interval $(\rho,\delta)$ does not contain an element of $C^{(n+1)}$. 
 Now, assume, towards a contradiction, the interval $(\rho,\delta)$ does not contain a  $C^{(n)}$-weakly shrewd cardinal that is not $C^{(n)}$-strongly unfoldable. The above observation then allow us to use Lemma \ref{lemma:separatePropertiesCn} to conclude that the interval $(\rho,\delta)$ does not contain  $C^{(n)}$-weakly shrewd cardinals.  
 This shows that  for every cardinal $\mu$ in the interval $(\rho,\delta)$, there exists a set-theoretic formula $\varphi(v_0,v_1)$, a cardinal $\mu<\theta\in C^{(n)}$ and a subset $z$ of $\mu$ such that $\varphi(z,\mu)$ holds in $\HH{\theta}$ and $\varphi(z\cap\bar{\mu},\bar{\mu})$ does not hold in $\HH{\bar{\theta}}$ for all $\bar{\theta}\in C^{(n)}$   and  all  cardinals $\bar{\mu}<\min(\mu,\bar{\theta})$. 
 We can then find a cardinal $\delta<\zeta\in C^{(1)}$ with the property that, in $\HH{\zeta}$,   for every cardinal $\mu$ in the interval $(\rho,\delta)$, there exists a set-theoretic formula $\varphi(v_0,v_1)$, a cardinal $\mu<\theta\in C^{(n)}$ and a subset $z$ of $\mu$ such that $\varphi(z,\mu)$ holds in $\HH{\theta}$ and $\varphi(z\cap\bar{\mu},\bar{\mu})$ does not hold in $\HH{\bar{\theta}}$ for all $\bar{\theta}\in C^{(n)}$   and  all  cardinals $\bar{\mu}<\min(\mu,\bar{\theta})$. 
 Let $\eta$ be the least cardinal with this property. Then the set $\{\eta\}$ is definable by a $\Sigma_2$-formula with parameters $\delta$, $\rho$ and $y$. We then directly know that the set $\{\eta\}$ is also definable by a $\Sigma_{m+1}$-formula with parameters $\rho$ and $y$.

 Now,  pick $\kappa<\vartheta\in C^{(m+1)}$ and use  Lemma \ref{lemma:WcnShrewdChar} to find  $\bar{\vartheta}\in C^{(m)}$, a cardinal $\bar{\kappa}<\min(\kappa,\bar{\vartheta})$, an elementary submodel $X$ of $\HH{\bar{\vartheta}}$ with $\bar{\kappa}+1\subseteq X$ and an elementary embedding $\map{j}{X}{\HH{\vartheta}}$ with $j\restriction\bar{\kappa}=\id_{\bar{\kappa}}$, $j(\bar{\kappa})=\kappa$ and $\rho,y\in\ran{j}$. Then $y\in X$ with $j(y)=y$. Moreover, we know that $\rho<\bar{\kappa}$ and therefore $j(\rho)=\rho$. In combination with the fact that all $\Sigma_{m+1}$-statements are upwards absolute from $X$ to $\VV$, this shows that $\delta,\eta\in X$ with $j(\delta)=\delta$ and $j(\eta)=\eta$. 
 Since $\bar{\kappa}$ is a cardinal in the interval $(\rho,\delta)$ in $\HH{\eta}$, we can now find a set-theoretic formula $\varphi(v_0,v_1)$, an ordinal $\bar{\kappa}<\theta\in X\cap\eta$ and $z\in\POT{\bar{\kappa}}\cap X$ with the property that, in $\HH{\eta}$, the ordinal $\theta$ is a cardinal in $C^{(n)}$ such that  $\varphi(z,\bar{\kappa})$ holds in $\HH{\theta}$ and $\varphi(z\cap\mu,\mu)$ does not hold in $\HH{\bar{\theta}}$ for all $\bar{\theta}\in C^{(n)}$   and  all  cardinals $\mu<\min(\bar{\kappa},\bar{\theta})$. 
 The fact that $\HH{\eta}\in X$ with $j(\HH{\eta})=\HH{\eta}$ then allows us to use the elementarity of $j$ to conclude that, in $\HH{\eta}$,  the ordinal $j(\theta)$ is a cardinal in $C^{(n)}$ with the property that  $\varphi(j(z),\kappa)$ holds in $\HH{j(\theta)}$ and $\varphi(j(z)\cap\mu,\mu)$ does not hold in $\HH{\bar{\theta}}$ for all $\bar{\theta}\in C^{(n)}$   and  all  cardinals $\mu<\min(\kappa,\bar{\theta})$. 
 But, this yields a contradiction, because  $\theta$ is an element of $C^{(n)}$ in $\HH{\eta}$, $\bar{\kappa}<\min(\kappa,\theta)$ and $j(z)\cap\bar{\kappa}=z$.  
\end{proof}

As a first application of Lemma \ref{lemma:Overspil}, we show that the principle \anf{\emph{$\Ord$ is essentially subtle}} can be axiomatized over the theory $\ZFC+\anf{\textit{$\Ord$ is essentially faint}}$ in a strong way.

\begin{proof}[Proof of Proposition \ref{proposition:Difference}]
 For each natural number $n>0$, let $\phi_n$ denote the sentence in the language of set theory that states that the existence of a proper class of $C^{(n)}$-weakly shrewd cardinals  implies the existence of a proper class of $C^{(n)}$-strongly unfoldable cardinals. Define $\mathrm{T}$ to be the theory consisting of these sentences. Then Theorem \ref{theorem:OrdSubtle1}  shows that $\mathrm{T}$ holds in every model of $\ZFC$ in which $\Ord$ is essentially subtle. Moreover, a combination of Theorem \ref{theorem:OrdSubtle1}   and Theorem \ref{theorem:CharOrdFaint} shows that $\Ord$ is essentially subtle in every model of $\ZFC$ in which $\Ord$ is essentially faint and $\mathrm{T}$ holds. Finally, assume that  we work in a model of $\ZFC$ in which the sentence $\phi_m$  fails for some natural number $m>0$. Then there exists a proper class of $C^{(m)}$-weakly shrewd cardinals and there are only boundedly many $C^{(m)}$-strongly unfoldable cardinals. In particular, there exists a proper class of $C^{(m)}$-weakly shrewd cardinals  that are not $C^{(m)}$-strongly unfoldable. In this situation, a combination of Lemmas \ref{lemma:separatePropertiesCn} and \ref{lemma:Overspil} shows that for every natural number $n>0$, there exists a proper class of $C^{(n)}$-weakly shrewd cardinals. Hence, we can apply Theorem \ref{theorem:CharOrdFaint} to conclude that $\Ord$ is essentially faint in this model.  
\end{proof}

By combing Theorem \ref{theorem:SubtleInHOD} with Proposition \ref{proposition:Difference}, we can now examine the conditions that allow the proposition's statement to be strengthened to result in finite axiomatizations.

\begin{proof}[Proof of Theorem \ref{theorem:DiffSentence}]
  First, assume that the theory $$\ZFC+\anf{\textit{There is a proper class of subtle cardinals}}+\anf{\textit{$\Ord$ is essentially subtle}}$$ is inconsistent and let $\phi$ denote the set-theoretic sentence  stating that only boundedly many cardinals are subtle in $\HOD$. 
 Now, if we work in a model $\VV$ of $\ZFC$ in which $\Ord$ is essentially subtle, then $\Ord$ is essentially subtle in $\HOD$ and, since our assumption ensures that $\HOD$ has only boundedly many subtle cardinals, we can conclude that $\phi$ holds in $\VV$.   Next, if we work in a model  of $\ZFC$ in which $\Ord$ is essentially faint and $\phi$ holds, then Theorem \ref{theorem:SubtleInHOD} directly shows that $\Ord$ is essentially subtle. In combination, this shows that \eqref{item:Diff2} implies \eqref{item:Diff3} in the statement  of the theorem. 
 
 Next, assume that there is a set-theoretic sentence $\phi$ satisfying the statements \eqref{item:Diff1.1} and \eqref{item:Diff1.2} listed in the theorem. Let  $\psi$ be the set-theoretic sentence that is given by the conjunction of $\phi$ and the sentence stating that there is a proper class of subtle cardinals. Then $${\ZFC\mathbin{~ + ~ }{\psi}}\mathbin{ ~ \vdash ~ }\anf{\textit{$\Ord$ is essentially subtle}}$$ and therefore Proposition \ref{proposition:OrdSubtleNotFinite} shows that $\ZFC\vdash\neg\psi$. In this situation, our assumptions imply that $${\ZFC\mathbin{~ + ~ }{\anf{\textit{$\Ord$ is essentially subtle}}}}\mathbin{ ~ \vdash ~ }\anf{\textit{There are only boundedly many subtle cardinals}}$$ and we can conclude that \eqref{item:Diff3} implies \eqref{item:Diff2} in the statement  of the theorem. 
 
 Finally, assume that there is a  set-theoretic sentence $\phi$  satisfying the statements \eqref{item:Diff1.1} and \eqref{item:Diff1.2} listed in the theorem. Let $\mathrm{T}$ be the theory given by Proposition \ref{proposition:Difference}.  We then know that $${\ZFC\mathbin{+}{\anf{\textit{$\Ord$ is essentially faint}}}\mathbin{+}\mathrm{T}}\mathbin{ ~ \vdash ~ }\phi$$ and hence there is a finite subtheory $\mathrm{F}$ of $\mathrm{T}$ with $${\ZFC\mathbin{+}{\anf{\textit{$\Ord$ is essentially faint}}}\mathbin{+}\mathrm{F}}\mathbin{ ~ \vdash ~ }\phi.$$ Let $\psi$ be the set-theoretic sentence that is given by the disjunction of $\phi$ with the conjunction of all sentences in $\mathrm{F}$. Our assumptions then ensure that  $\psi$ satisfies the statement \eqref{item:Diff1.1} of the theorem. Moreover, the definition of $\psi$ implies that $${{\ZFC\mathbin{+}{\anf{\textit{$\Ord$ is essentially faint}}}}\mathbin{+}{\psi}}\mathbin{ ~ \vdash ~ }   \phi$$ and hence our assumption ensure that $\psi$ satisfies  statement \eqref{item:Diff1.2} of the theorem. Finally, assume that we work in a model of $\ZFC$ in which $\psi$ does not hold. Then there exists a sentence in $\mathrm{F}$ that does not hold. In this situation, the properties of the theory $\mathrm{T}$ allow us to conclude that $\Ord$ is essentially faint in this model. This shows that $\psi$ also satisfies  statement \eqref{item:Diff1.3} of the theorem. In combination, these arguments show that  \eqref{item:Diff1} implies \eqref{item:Diff2} in the statement  of the theorem and therefore these assumptions are equivalent.  
\end{proof}

We end this section by using Lemma \ref{lemma:Overspil} to show that the principle \anf{\emph{$\Ord$ is essentially faint}}  does not imply the existence of strongly inaccessible cardinals.

\begin{proof}[Proof of Theorem \ref{MAIN:Equicons}]
 Work in a model of $\ZFC + {\VV=\LL}$ in which a proper class of subtle cardinals exists and no inaccessible cardinal is a limit of subtle cardinals. 
 Let $S$ denote the class of all subtle cardinals. Define $R$ to be the unique class function with domain $S$ that satisfies $R(\min(S))=\omega$ and $R(\delta)=\sup(S\cap\delta)^+$ for all $\min(S)<\delta\in S$. Our assumptions then ensure that $R$ is a well-defined regressive function.  
  Define $\PPP$ to be the class partial order that is given by the Easton support product of all partial orders of the form $\Add{R(\delta)}{\delta}$ with $\delta\in S$ and  let $G$ be $\PPP$-generic over $\VV$. 
  Then standard  arguments (as in {\cite[pp. 235--236]{MR1940513}}) show that $\VV[G]$ is a model of $\ZFC$ with the same regular cardinals as $\VV$.

 \begin{claim*}
   There are no inaccessible cardinals in $\VV[G]$. 
 \end{claim*}
  
  \begin{proof}[Proof of the Claim]
   Assume, towards a contradiction, that there is an inaccessible $\kappa$ in $\VV[G]$. 
   Set $\delta=\min(S\setminus\kappa)\geq\kappa$. Since $\kappa$ is not a limit of subtle cardinals in $\VV$, we know that $R(\delta)<\kappa$. 
   Moreover, our setup ensures that $$\big(2^{R(\delta)}\big)^{\VV[G]} ~ \geq ~ \delta ~ \geq ~ \kappa$$ holds, contradicting the inaccessibility of $\kappa$ in $\VV[G]$.  
  \end{proof}

  \begin{claim*}
   In $\VV[G]$, there is a proper class of weakly shrewd cardinals. 
  \end{claim*}

  \begin{proof}[Proof of the Claim]
      Fix an ordinal $\alpha>\omega$ and pick  $\delta\in S$ with $R(\delta)>\alpha$.   Set $S_0=S\cap\delta$ and let $\QQQ$ denote the Easton support product of all partial orders of the form $\Add{R(\gamma)}{\gamma}$ with $\gamma\in S_0$, as constructed in $\VV$. In addition, let $\RRR$ denote the partial order $\Add{R(\delta)}{\delta}$ as constructed in $\VV$. 
 A standard factor analysis of the class partial order $\PPP$ (again, as in {\cite[pp. 235--236]{MR1940513}}) now allows us to find an inner model $M$ of $\VV[G]$ such that  $\VV\subseteq M$, $({}^\delta\VV)^M\subseteq\VV$ and there    are $G_0,G_1\in\VV[G]$ such that $G_0\times G_1$ is $(\QQQ\times\RRR)$-generic over $M$ with $\VV[G]=M[G_0,G_1]$.   
   Then $\delta$ is a subtle cardinal in $M$ and   the definition of $R$ in $\VV$ ensures that $\delta$ is the minimal subtle cardinal above $R(\delta)$ in $M$.     In particular, it follows that the set $\{\delta\}$ is definable by a $\Sigma_2$-formula with parameter $R(\delta)$ in $M$.

   An application of   Lemma \ref{lemma:ShrewdFromSubtle} now shows that, in $M$, the interval $(R(\delta),\delta)$ contains an inaccessible  weakly shrewd cardinal $\kappa$. 
      In the following, fix a cardinal $\theta>\delta$ that is an element of $C^{(1)}$ in $\VV[G]$ and an element $z$ of $\HH{\theta}^{\VV[G]}$. 
   Pick a $(\QQQ\times\RRR)$-name $\dot{z}$ in $M$ with $z=\dot{z}^{G_0\times G_1}$ and an ordinal $\vartheta>\theta$ that is an element of $C^{(2)}$ in $M$ and satisfies $\dot{z}\in\HH{\vartheta}^M$. 
   Proposition \ref{proposition:BasicWeaklyShrewd} and Lemma \ref{lemma:WcnShrewdChar} now show that, in $M$, we can find 
   a cardinal $\bar{\vartheta}\in C^{(1)}$, 
   a cardinal $\bar{\kappa}<\min(\kappa,\bar{\vartheta})$, 
    an elementary submodel $X$ of $\HH{\bar{\vartheta}}$ of cardinality $\bar{\kappa}$ with $\bar{\kappa}+1\subseteq X$ and 
    an elementary embedding $\map{j}{X}{\HH{\vartheta}}$ with $j\restriction\bar{\kappa}=\id_{\bar{\kappa}}$, $j(\bar{\kappa})=\kappa$  and $R(\delta),\theta,\dot{z}\in\ran{j}$. 
    We then know that $R(\delta)<\bar{\kappa}$ and this directly implies that $\delta\in X$ with $j(\delta)=\delta$. In particular, it follows that $\bar{\kappa}<\kappa<\delta<\bar{\vartheta}$. Since the partial order $\QQQ\times\RRR$ has cardinality $\delta$ in $M$, we know that the cardinals $\vartheta$ and $\bar{\vartheta}$ are elements of $C^{(1)}$ in both $M[G_0]$ and $\VV[G]$. 
    Moreover, since  $X$ has cardinality $\bar{\kappa}<\delta$ in $M$, there is a permutation $\sigma$ of $\delta$ in $M$  that extends  the injection  $\map{j\restriction(X\cap\delta)}{X\cap\delta}{\delta}$.  
   Since $\RRR=\Add{R(\delta)}{\delta}^M$, this  shows that $M$ also contains an automorphism $\tau$ of  $\RRR$ that is induced by  the action of $\sigma$ on the supports of the conditions in $\Add{R(\delta)}{\delta}$,\footnote{Here, we view conditions in partial orders of the form $\Add{\mu}{\nu}$ as  functions $p$ of cardinality less than $\mu$ with $\dom{p}\subseteq\nu$ and $p(\gamma)\in\Add{\mu}{1}$ for every $\gamma\in\dom{p}$.}  
     in the sense that $\dom{\tau(p)}=\sigma[\dom{p}]$ and $\tau(p)(\sigma(\gamma))=p(\gamma)$ holds for every condition $p$ in $\Add{R(\delta)}{\delta}$ and all $\gamma\in\dom{p}$. 
   Since      our setup ensures that  $\bar{\kappa}$ is inaccessible in $M$ and $\HH{\bar{\kappa}}^M\subseteq X$ with $j\restriction\HH{\bar{\kappa}}^M=\id_{\HH{\bar{\kappa}}^M}$, it follows that  $\Add{R(\delta)}{1}^M\cup\{\Add{R(\delta)}{1}^M\}\subseteq X$ and $j\restriction\Add{R(\delta)}{1}^M=\id_{\Add{R(\delta)}{1}^M}$. 
   Moreover, we also know that $\RRR\in X$ with $j(\RRR)=\RRR$. 
   Finally, our definitions ensure that  $\tau(p)=j(p)$ holds for every condition $p$ in $\RRR$ that is an element of $X$.

  Next, since $\POT{\QQQ}\in\HH{\bar{\kappa}}^M\subseteq X$, we know that $G_0$ is $\QQQ$-generic over $X$, 
  $\HH{\vartheta}^M[G_0]=\HH{\vartheta}^{M[G_0]}$ and 
  the model $X[G_0]$ is an elementary submodel of $\HH{\bar{\vartheta}}^{M[G_0]}$. 
  Moreover, the fact that  $j\restriction\QQQ=\id_{\QQQ}$ implies that   there is a canonical lifting $$\map{j_0}{X[G_0]}{\HH{\vartheta}^{M[G_0]}}$$ of $j$ to $X[G_0]$. 
   Now, set $G_1^\prime=\tau^{{-}1}[G_1]\subseteq\RRR$. Since  $\tau$ is an element of $\HH{\vartheta}^M$, it follows that $G_1^\prime$ is $\RRR$-generic over $M[G_0]$ with $$\HH{\vartheta}^{M[G_0]}[G_1^\prime] ~ = ~ \HH{\vartheta}^{M[G_0]}[G_1] ~ = ~ \HH{\vartheta}^{\VV[G]}.$$ 
   Set $G_1^{\prime\prime}=G^\prime_1\cap X$.  
     Since the partial order $\QQQ$ has cardinality less than $\bar{\kappa}$ in $M$ and the partial order $\RRR=\Add{R(\delta)}{\delta}$ satisfies the $\bar{\kappa}$-chain condition in $M$, if follows that   $\RRR$ satisfies the $\bar{\kappa}$-chain condition in $M[G_0]$. 
 Together with the fact that  $\RRR\cap X=\RRR\cap X[G_0]$, this  ensures that $G_1^{\prime\prime}$ is $\RRR$-generic over $X[G_0]$ and $X[G_0,G_1^{\prime\prime}]$ is an elementary submodel of $\HH{\bar{\vartheta}}^{\VV[G]}$.  
   Moreover, since we have ensured that $j_0[G_1^{\prime\prime}]=j[(\tau^{{-}1}[G_1])\cap X]\subseteq G_1$, there is a canonical lifting $$\map{j_1}{X[G_0,G_1^{\prime\prime}]}{\HH{\vartheta}^{\VV[G]}}$$ of $j_0$ to $X[G_0,G_1^{\prime\prime}]$. 
  Pick $\bar{\theta}\in X$ with $j(\bar{\theta})=\theta$. Since $\bar{\vartheta}$ is an element of $C^{(1)}$ in $\VV[G]$, elementarity implies that $\bar{\theta}$ is also an element of this class. Set $Y=\HH{\bar{\theta}}^{\VV[G]}\cap X[G_0,G_1^{\prime\prime}]$ and define $i=j_1\restriction Y$. 
  Since $\HH{\bar{\theta}}^{\VV[G]}\in X[G_0,G_1^{\prime\prime}]$,  elementarity ensures that the set $Y$ is an elementary submodel of $\HH{\bar{\theta}}^{\VV[G]}$. Moreover, our setup ensures that $\map{i}{Y}{\HH{\theta}^{\VV[G]}}$ is an elementary embedding with $i\restriction\bar{\kappa}=\id_{\bar{\kappa}}$, $i(\bar{\kappa})=\kappa$, $i(\delta)=\delta$ and $z\in\ran{i}$. 
   These computations show that, in $\VV[G]$, the cardinal $\kappa>\alpha$ is  weakly shrewd. 
  \end{proof}

 In combination, the above claims show that $\VV[G]$ is a model of $\ZFC$ that contains a proper class of weakly shrewd cardinals and in which no strongly inaccessible cardinals exist.

 Now, let $\phi$ denote the sentence in the language of set-theory stating that there are no strongly inaccessible cardinals and there is a proper class of weakly shrewd cardinals. Then $\phi$ implies that there is a proper class of weakly shrewd cardinals that are not strongly unfoldable and, by Lemma \ref{lemma:Overspil}, this conclusion ensures  that for every natural number $n>0$, there is a proper class of $C^{(n)}$-weakly shrewd cardinals. Therefore,  Theorem \ref{theorem:CharOrdFaint}  shows that   $${\ZFC\mathbin{+}\phi}\mathbin{\vdash}{\anf{\textit{$\Ord$ is essentially faint}}\mathbin{+}\anf{\textit{There are no strongly inaccessible cardinals}}}.$$ In addition, the above computations show that the consistency of the theory $$\ZFC\mathbin{+}\anf{\textit{There is a proper class of subtle cardinals}}$$ implies the consistency of $\ZFC+\phi$. Hence, we know that   \eqref{item:EquiconsNonInacc1} implies \eqref{item:EquiconsNonInacc3} in the statement  of the theorem.

 Next, if we work in a model of $\ZFC$ in which $\Ord$ is essentially faint and there are no strongly inaccessible cardinals, then Theorem \ref{theorem:OrdSubtle1} shows that $\Ord$ is not essentially subtle  and Theorem \ref{theorem:SubtleInHOD} allows us to conclude that a proper class of cardinals in subtle in $\HOD$. In particular, we know that  \eqref{item:EquiconsNonInacc2} implies \eqref{item:EquiconsNonInacc1} in the statement  of the theorem. 

  This completes the proof of the theorem, because the remaining implication from   \eqref{item:EquiconsNonInacc3} to \eqref{item:EquiconsNonInacc2} in its statement is trivial.  
\end{proof}


\section{Abstract logics}\label{section:AbstractLogics}

In the remainder of this paper, we will connect the theory developed above with the study of structural properties of abstract logics.  
We start by recalling the relevant definitions. Our setup is based on the definitions given in {\cite[Section 2.5]{MR1059055}} and our presentation closely follows those provided in \cite{MR4093885}, \cite{bdgm} and \cite{outwardcompactness}.

\begin{definition}\label{definition:StructuresAndRenamings}
  \begin{enumerate}
    \item A \emph{language}  is a tuple $$\tau ~ = ~ \langle\CCCC_\tau,\FFFF_\tau,\RRRR_\tau,\aaaa_\tau\rangle,$$ where $\CCCC_\tau$,  $\FFFF_\tau$ and $\RRRR_\tau$ are pairwise disjoint sets and $\map{\aaaa_\tau}{\FFFF_\tau\cup\RRRR_\tau}{\omega\setminus\{0\}}$ is a function.  We then call $\CCCC_\tau$ the \emph{set of constant symbols} of $\tau$, $\FFFF_\tau$ the \emph{set of function symbols} of $\tau$, $\RRRR_\tau$ the \emph{set of relation symbols} of $\tau$, $\aaaa_\tau$  the \emph{arity function} of $\tau$ and $\SSSS_\tau=\CCCC_\tau\cup\FFFF_\tau\cup\RRRR_\tau$ the set of \emph{logical symbols} of $\tau$.

    \item Given a language $\tau$, a \emph{$\tau$-structure} is a tuple $$M ~ = ~ \langle\dom{M},\seq{c^M}{c\in\CCCC_\tau}, \seq{f^M}{f\in\FFFF_\tau},\seq{r^M}{r\in\RRRR_\tau}\rangle,$$  where $\dom{M}$ is a non-empty set, each $c^M$ is an element of $\dom{M}$,     each $f^M$ is an $\aaaa_\tau(f)$-ary function on $\dom{M}$ and each $r^M$ is an $\aaaa_\tau(r)$-ary relation on $\dom{M}$. 
    
      \item Given a language $\tau$,  we let $\Str_\tau$ denote the class of all $\tau$-structures.  
      
      \item Given a language $\tau$ and $M,N\in\Str_\tau$, an \emph{isomorphism} between $M$ and $N$ is a bijection $\map{\pi}{\dom{M}}{\dom{N}}$ satisfying: 
       \begin{enumerate}
         \item If   $c\in\CCCC_\tau$, then $\pi(c^M)=c^N$. 
         
         \item If $f\in\FFFF_\tau$ and $x_0,\ldots,x_{\aaaa_\tau(f)-1}\in\dom{M}$, then $$\pi(f^M(x_0,\ldots,x_{\aaaa_\tau(f)-1})) ~ = ~ f^N(\pi(x_0),\ldots,\pi(x_{\aaaa_\tau(f)-1})).$$ 
         
         \item If $r\in\RRRR_\tau$ and $x_0,\ldots,x_{\aaaa_\tau(r)-1}\in\dom{M}$, then $$r^M(x_0,\ldots,x_{\aaaa_\tau(r)-1}) ~ \Longleftrightarrow ~ r^N(\pi(x_0),\ldots,\pi(x_{\aaaa_\tau(r)-1})).$$  
       \end{enumerate}
      
       \item Given a language $\tau$ and $M,N\in\Str_\tau$, the structure $M$ is a \emph{substructure} of the structure $N$ if $\dom{M}\subseteq\dom{N}$, $c^M=c^N$ for all $c\in\CCCC_\tau$, $f^M=f^N\restriction\dom{M}^{\aaaa_\tau(f)}$ for all $f\in\FFFF$ and $r^M=r^N\cap\dom{M}^{\aaaa_\tau(r)}$ for all $r\in\RRRR$. 
      
      \item A language $\sigma$ is a \emph{sublanguage} of a language $\tau$ if $\mathfrak{C}_\sigma\subseteq\mathfrak{C}_\tau$, $\mathfrak{F}_\sigma\subseteq\mathfrak{F}_\tau$, $\mathfrak{R}_\sigma\subseteq\mathfrak{R}_\tau$ and $\aaaa_\sigma=\aaaa_\tau\restriction(\mathfrak{F}_\sigma\cup\mathfrak{R}_\sigma)$.  
      
      \item Given a sublanguage $\sigma$ of a language $\tau$ and $M\in\Str_\tau$, the \emph{$\sigma$-reduct} $M\mathbin{\restriction}\sigma$ of $M$ is the unique element $N$ of $\Str_\sigma$ with $\dom{N}=\dom{M}$,  $c^N=c^M$ for all $c\in\CCCC_\sigma$, $f^N=f^M$ for all $f\in\FFFF_\sigma$ and $r^N=r^M$ for all $r\in\RRRR_\sigma$.

    \item A \emph{renaming} of a language $\sigma$ into a language $\tau$ is a bijection $$\map{u}{\SSSS_\sigma}{\SSSS_\tau}$$ satisfying $u[\mathfrak{C}_\sigma]=\mathfrak{C}_\tau$, $u[\mathfrak{F}_\sigma]=\mathfrak{F}_\tau$, $u[\mathfrak{R}_\sigma]=\mathfrak{R}_\tau$ and $\aaaa_\sigma(s)=\aaaa_\tau(u(s))$ for all $s\in\mathfrak{F}_\sigma\cup\mathfrak{R}_\sigma$. 
    
    \item Given a renaming $u$ of a language $\sigma$ into a language $\tau$,   we let $\map{u^*}{\Str_\sigma}{\Str_\tau}$ denote the   bijective class function with the property that for all $M\in \Str_\sigma$, we have $\dom{u^*(M)}=\dom{M}$ and  $u(s)^{u^*(M)}=s^M$ for all $s\in\SSSS_\sigma$.   
  \end{enumerate}
 \end{definition}

  Note that, if $u$ is a renaming of a language $\sigma$  into a language $\tau$, then the inverse $u^{{-}1}$ of $u$ is a renaming of $\tau$ into $\sigma$ with the property that $(u^{{-1}})^*=(u^*)^{{-}1}$ holds.

  We are now ready to introduce the notion of an abstract logic:

 \begin{definition}\label{definition:AbstractLogic}
  An \emph{abstract logic}  consists of a class function $\calL$ and a binary class relation $\models_\calL$ satisfying  the following statements: 
     \begin{enumerate}
      \item The domain of $\calL$ is the class of all languages. 
      
      \item If $M\models_\calL\phi$ holds, then there exists a language $\tau$ such that $M\in\Str_\tau$ and $\phi\in\calL(\tau)$.  
      
      \item  (Monotonicity) If $\sigma$ is a sublanguage of $\tau$, then $\calL(\sigma)\subseteq\calL(\tau)$. 
      
      \item\label{item:AbstractLogicExpansion} (Expansion) If $\sigma$ is a sublanguage of $\tau$, $M\in\Str_\tau$ and $\phi\in\calL(\sigma)$, then $$M\models_\calL\phi ~ \Longleftrightarrow ~ {M\mathbin{\restriction}\sigma}\models_\calL\phi.$$
      
       \item\label{item:AbstractLogic-Iso} (Isomorphism) Given a language $\tau$ and isomorphic $\tau$-structures $M$ and $N$, we have $$M\models_\calL\phi ~ \Longleftrightarrow ~ N\models_\calL\phi$$ for all $\phi\in\calL(\tau)$.  
       
      \item\label{item:AbstractLogic-renaming}  (Renaming) Every renaming $u$ of a language $\sigma$ into a language $\tau$ induces a unique bijection $\map{u_*}{\calL(\sigma)}{\calL(\tau)}$ with the property that $$M\models_\calL\phi ~ \Longleftrightarrow ~ u^*(M)\models_\calL u_*(\phi)$$  holds for every $M\in\Str_\sigma$ and every $\phi\in\calL(\sigma)$.  
            
     \item\label{item:occurrenceNumber} (Occurrence number) There exists an infinite cardinal $\ooo$   with the property that for every language $\tau$ and every $\phi\in\calL(\tau)$, there exists a sublanguage $\sigma$ of $\tau$ with $\betrag{\SSSS_\sigma}<\ooo$ and $\phi\in\calL(\sigma)$. The least such cardinal is called the \emph{occurrence number} of the logic. 
     \end{enumerate}
  \end{definition}

  In the following, we just write $\calL$ to denote an abstract logic. We then call the corresponding relation $\models_\calL$ the \emph{satisfaction relation} of $\LL$ and, given a language $\tau$,  we say that $\calL(\tau)$ is the set of \emph{$\tau$-sentences}. In addition, we write $\ooo_\calL$ to denote the occurrence number of an abstract logic $\calL$. 
  Note that the uniqueness of $u_*$ in Item \eqref{item:AbstractLogic-renaming} of Definition \ref{definition:AbstractLogic} implies that $(u^{{-}1})_*=(u_*)^{{-}1}$ holds for every renaming $u$.

  The above definitions now allow us to generalize fundamental concepts from first-order logic to all abstract logics.

  \begin{definition}\label{definition:AbstractLogicNotions}
   Let $\calL$ be an abstract logic. 
   \begin{enumerate}
    \item An \emph{$\calL$-theory} is a set  $T$ with the property that $T\subseteq\calL(\tau)$ for some  language $\tau$. 
    
    \item An $\calL$-theory $T$ is \emph{satisfiable} if there is a language $\tau$ with $T\subseteq\calL(\tau)$ and $M\in\Str_\tau$ with $M\models_\calL\phi$ for all $\phi\in T$. 
    
    \item  Given an  infinite cardinal $\kappa$, an  $\calL$-theory $T$ is \emph{${<}\kappa$-satisfiable}  if every subset of $T$ of cardinality less than $\kappa$ is satisfiable. 
    
    \item An infinite cardinal $\kappa$ is a \emph{strong compactness cardinal for $\calL$} if every ${<}\kappa$-satisfiable $\calL$-theory is satisfiable. 
    
    \item An infinite cardinal $\kappa$ is a \emph{weak compactness cardinal for $\calL$} if every ${<}\kappa$-satisfiable $\calL$-theory of cardinality $\kappa$ is satisfiable.  
    
    \item Given a language $\tau$ and $N\in\Str_\tau$, an \emph{$\calL$-elementary substructure} of $N$ is a substructure $M$ of $N$ with the property that $$M\models\phi ~ \Longleftrightarrow ~ N\models\phi$$ holds for all $\phi\in\calL(\tau)$. 
    
    \item An uncountable cardinal $\kappa$ is a \emph{L\"owenheim--Skolem--Tarski number} for $\calL$ if for every language $\tau$ with  $\betrag{\SSSS_\tau}<\kappa$ and every $\tau$-structure $N$, there exists an $\calL$-elementary substructure $M$ of $N$ with $\betrag{\dom{M}}<\kappa$.  
    
    \item\label{item:WeakLST} An uncountable cardinal $\kappa$ is a \emph{weak L\"owenheim--Skolem--Tarski number} for $\calL$ if for every language $\tau$ with the property that $\betrag{\SSSS_\tau}<\kappa$ and every $\tau$-structure $N$ with $\betrag{\dom{N}}=\kappa$, there exists an $\calL$-elementary substructure $M$ of $N$ with $\betrag{\dom{M}}<\kappa$.  
    
       \item\label{item:StrictkLST} An uncountable cardinal $\kappa$ is a \emph{strict L\"owenheim--Skolem--Tarski number} for $\calL$ if for every language $\tau$ with the property that $\betrag{\SSSS_\tau}<\kappa$, every $\tau$-structure $N$ with $\betrag{\dom{N}}=\kappa$ and every $\phi\in\calL(\tau)$ with $N\models_\calL\phi$, there exists a  substructure $M$ of $N$ with $\betrag{\dom{M}}<\kappa$ and $M\models_\calL\phi$.  
   \end{enumerate}
  \end{definition}

  In {\cite[Section 4]{bdgm}}, it is verified that  Makowsky's Theorem \ref{theorem:Makowsky} holds true with respect to the above definitions. 
  In another direction, an unpublished result of Stavi shows that the validity of Vop\v{e}nka's Principle is also equivalent to the fact that  analogs of the Downward L\"owenheim--Skolem  Theorem hold for all abstract logics (see {\cite[Theorem 6]{MR2833152}}). For sake of completeness and to motivate the notions introduced in Items \eqref{item:WeakLST} and \eqref{item:StrictkLST} of Definition \ref{definition:AbstractLogicNotions}, we present a proof of Stavi's result.

 \begin{theorem}[Stavi]\label{theorem:Stavi}
  The following schemes are equivalent over $\ZFC$: 
   \begin{enumerate}
   \item\label{item:StaviVP} Vop\v{e}nka's Principle. 
   
   \item\label{item:StaviLSTNumber} Every abstract logic  has a L\"owenheim--Skolem--Tarski number. 
   
   \item\label{item:StaviSTRICLST} For every abstract logic $\calL$, there exists an uncountable cardinal $\kappa$ with the property that for every language $\tau$ with $\betrag{\SSSS_\tau}<\kappa$, every $\tau$-structure $N$  and every $\phi\in\calL(\tau)$ with $N\models_\calL\phi$, there exists a  substructure $M$ of $N$ with $\betrag{\dom{M}}<\kappa$ and $M\models_\calL\phi$.  
   \end{enumerate}
 \end{theorem}
   
 \begin{proof} 
  First, assume that \eqref{item:StaviVP} holds and fix an abstract logic $\calL$. Then there is a natural number $n>0$, a $\Sigma_n$-formula $\varphi(v_0,v_1,v_2)$, a $\Sigma_n$-formula $\psi(v_0,v_1,v_2)$ and a set $z$ with the property that the  function $\calL$ is defined by the formula $\varphi(v_0,v_1,v_2)$ and the parameter $z$ and the relation $\models_\calL$ is defined by the formula $\psi(v_0,v_1,v_2)$ and the parameter $z$. 
  By {\cite[Corollary 4.15]{zbMATH06029795}}, we can now find a \emph{$C^{(n)}$-extendible cardinal} $\kappa$ ({i.e.,} a cardinal $\kappa$ with the property that for every $\eta>\kappa$, there exists an ordinal $\zeta$ and an elementary embedding $\map{j}{\VV_\eta}{\VV_\zeta}$ with $\crit{j}=\kappa$, $j(\kappa)>\eta$ and $j(\kappa)\in C^{(n)}$) with   $z\in\HH{\kappa}$. 
   Now, fix a language $\sigma$ with $\betrag{\SSSS_\sigma}<\kappa$ and a $\sigma$-structure $K$. Then there exists a renaming $u$ of $\sigma$ into a language $\tau\in\VV_\kappa$. Since {\cite[Proposition 3.4]{zbMATH06029795}} shows that $\kappa\in C^{(n+2)}$, we know that $\calL(\tau)$ is also an element of $\VV_\kappa$. 
    Set  $M=u^*(K)$. Pick $\kappa<\eta\in C^{(n+2)}$ with $M\in\VV_\eta$. Then, there is an ordinal $\zeta$ and an elementary embedding $\map{j}{\VV_\eta}{\VV_\zeta}$ with $\crit{j}=\kappa$, $j(\kappa)>\eta$ and $j(\kappa)\in C^{(n)}$. 
   Since $\tau\in\VV_\kappa$ and $j(\tau)=\tau$, we now know that $j(M)$ is a $\tau$-structure with $s^{j(M)}=j(s^M)$ for every $s\in\SSSS_\tau$. This shows that $j(M)$ has a substructure $N$ with $\dom{N}=j[\dom{M}]$. We then know that the map $j\restriction\dom{M}$ is an isomorphism between $M$ and $N$ and this isomorphism is an element of $\VV_\zeta$. 
   In this situation, 
 the correctness properties of $\VV_\eta$,  the elementarity of $j$  and    Item \eqref{item:AbstractLogic-Iso} in Definition \ref{definition:AbstractLogic}     ensure that  
   \begin{equation}\label{equation:IsoInVZeta}
     \VV_\zeta ~ \models ~ \psi(M,\phi,z)\longleftrightarrow\psi(N,\phi,z) 
    \end{equation} 
     holds for all $\phi\in\calL(\tau)$. 
   Moreover, since $\kappa$ and $\eta$ are both elements of $C^{(n)}$, we know that the formula $\psi(v_0,v_1,v_2)$ is absolute between $\VV_\kappa$ and $\VV_\eta$. Elementarity then ensures that this formula is also absolute between $\VV_{j(\kappa)}$ and $\VV_\zeta$. Moreover, since $j(\kappa)$ is also an element of $C^{(n)}$, it follows that $\psi(v_0,v_1,v_2)$ is absolute between $\VV_\eta$ and $\VV_\zeta$.   
    For all $\phi\in\calL(\tau)$, we then have 
   $$\VV_\zeta\models\psi(j(M),\phi,z) ~  \Longleftrightarrow ~ \VV_\eta\models\psi(M,\phi,z) ~ \Longleftrightarrow ~ \VV_\zeta\models\psi(M,\phi,z) ~ \Longleftrightarrow ~ \VV_\zeta\models\psi(N,\phi,z),$$ where the first equivalence is given by elementarity of $j$, the second equivalence follows from the above absoluteness considerations and the third equivalence is given by \eqref{equation:IsoInVZeta}. 
  The elementarity of $j$ now yields a substructure $S$ of $M$ with $\betrag{\dom{S}}<\kappa$ and the property that $$\VV_\eta ~ \models  ~ \psi(M,\phi,z)\longleftrightarrow\psi(S,\phi,z)$$ holds for all $\phi$ in $\calL(\tau)$. Since $\eta$ is an element of $C^{(n)}$, this shows that $S$ is an $\calL$-elementary submodel of $M$. 
   Set $R=(u^*)^{{-}1}(M)\in\Str_\sigma$. Then $R$ is a submodel of $K$ with $\betrag{\dom{R}}<\kappa$. Moreover,  we can apply Item \eqref{item:AbstractLogic-renaming} in Definition \ref{definition:AbstractLogic} to conclude that $R$ is an $\calL$-elementary submodel of $K$. 
We can conclude that $\kappa$ is a L\"owenheim--Skolem--Tarski number for $\calL$. These computations show that \eqref{item:StaviVP} implies \eqref{item:StaviLSTNumber}.

  Now, assume that \eqref{item:StaviSTRICLST} holds and    fix a proper class $\Gamma$ of graphs. 
  Then there is an abstract logic $\calL_\Gamma$ such that the following statements hold: 
  \begin{itemize}
    \item  For every language $\tau$, the set $\calL_\Gamma(\tau)$ consists of all first-order sentences of $\tau$ together with a distinguished sentence $\phi_r$ for every binary relation symbol $r$ in $\tau$. 
    
    \item Given  is a language $\tau$, $M\in\Str_\tau$ and  a first-order sentence $\phi\in\calL_\Gamma(\tau)$,     then $M\models_{\calL_\Gamma}\phi$ holds if and only if $\phi$ holds in $M$ (in the sense of first-order structures). 
    
    \item Given  is a language $\tau$, $M\in\Str_\tau$ and  a binary relation symbol $r$ in $\tau$,     then $M\models_{\calL_\Gamma}\phi_r$ holds if and only if $\langle\dom{M},r^M\rangle$ is a graph that is isomorphic to a graph in $\Gamma$. 
  \end{itemize}

  Our assumption now yields an  uncountable cardinal $\kappa$ with the property that for every language $\tau$ with $\betrag{\SSSS_\tau}<\kappa$, every $\tau$-structure $N$  and every $\phi\in\calL_\Gamma(\tau)$ with $N\models_{\calL_\Gamma}\phi$, there exists a  substructure $M$ of $N$ with $\betrag{\dom{M}}<\kappa$ and $M\models_{\calL_\Gamma}\phi$.    
  Since $\Gamma$ is a proper class, we can fix a graph $G$ in $\Gamma$ that has cardinality greater than $\kappa$. 
  Let $\sigma$ denote a language with a single binary relation symbol $e$ and let $\tau$ be a language extending $\sigma$ that adds functions symbols for first-order Skolem functions to $\sigma$. View $G$ as a $\sigma$-structure and  let $G^\prime$ be a $\tau$-structure whose $\sigma$-reduct  is $G$ and that interprets the new function symbols as the corresponding first-order Skolem functions. 
  Then Item \eqref{item:AbstractLogicExpansion} in Definition \ref{definition:AbstractLogic} ensures that $G^\prime\models_{\calL_\Gamma}\phi_e$ holds and, since $\betrag{\SSSS_\tau}<\kappa$,  we can use our assumption to find a substructure $H^\prime$ of $G^\prime$ with $\betrag{H^\prime}<\kappa$ and $H^\prime\models_{\calL_\Gamma}\phi_e$.  
   Let $H$ denote the $\sigma$-reduct of $H^\prime$. Our setup now ensures that  $H$ is a first-order elementary submodel of $G$ and $H\models_{\calL_\Gamma}\phi_e$ holds. This shows that  $\langle\dom{H},e^H\rangle$  is a graph that is isomorphic to an element $K$ of $\Gamma$. Since $\dom{H}$ has cardinality less than $\kappa$, we know that $K\neq G$ and $K$ embeds into $G$. These computations show that \eqref{item:StaviVP} holds. 
  
 This completes the proof of the theorem, because  \eqref{item:StaviLSTNumber} obviously implies \eqref{item:StaviSTRICLST}. 
 \end{proof}

Motivated by Theorem \ref{theorem:Stavi}, we will also use the developed theory to  prove the following analogs of Theorems \ref{MAIN:EssentiallySubtle} and \ref{MAIN:EssentiallyFaint} for strict L\"owenheim--Skolem--Tarski numbers in the next section:

\begin{theorem}\label{MAIN:StrictLST}
 \begin{enumerate}
  \item\label{item:StrictLST1} The following schemes of sentences are equivalent over $\ZFC$: 
 \begin{enumerate}
  \item $\Ord$ is essentially faint. 
  
  \item Every abstract logic has a strict L\"owenheim--Skolem--Tarski number.  
 \end{enumerate}

  \item  The following schemes of sentences are equivalent over $\ZFC$: 
    \begin{enumerate}
      \item $\Ord$ is essentially subtle. 
  
     \item Every abstract logic has a stationary class of strict L\"owenheim--Skolem--Tarski numbers. 
     
      \item Every abstract logic has a stationary class of weak L\"owenheim--Skolem--Tarski numbers. 
    \end{enumerate}
  \end{enumerate}
\end{theorem}

In the last part of the next section, we will  consider  the question of whether the first equivalence can also be extended to characterize the existence of  weak L\"owenheim--Skolem--Tarski numbers for all abstract logics.


\section{Structural properties of abstract logics}\label{section:ProofsCompactness}

 This section contains the proofs of Theorems \ref{MAIN:EssentiallySubtle}, \ref{MAIN:EssentiallyFaint} and \ref{MAIN:StrictLST}.   
  The following lemma shows how strong structural properties of abstract logics can be derived from the existence of $C^{(n)}$-weakly shrewd cardinals.

\begin{lemma}\label{lemma:CnWeaklyShrewd}
 Let $n>0$ be a natural number, let $\varphi(v_0,v_1,v_2)$ and $\psi(v_0,v_1,v_2)$ be $\Sigma_n$-formulas and let $z$ be a set with the property that there is an abstract logic $\calL$ such that the  function $\calL$ is defined by the formula $\varphi(v_0,v_1,v_2)$ and the parameter $z$ and the relation $\models_\calL$ is defined by the formula $\psi(v_0,v_1,v_2)$ and the parameter $z$. Then the following statements hold: 
 \begin{enumerate}
  \item\label{item:WeakCompactFromWShrewd1} There exists a cardinal $\mu\in C^{(1)}$ such that the following statements hold: 
   \begin{enumerate}
    \item $\mu>\ooo_\calL$ and $z\in\HH{\mu}$. 
    
    \item If $\sigma$ is a language that is an element of $\HH{\ooo_\calL}$, then $\calL(\sigma)\in\HH{\mu}$. 
    
    \item  If  $\sigma$ is a language that is an element of $\HH{\ooo_\calL}$ and $\pi$ is a non-trivial permutation of $\calL(\sigma)$, then  there exists a $\sigma$-structure $M_{\sigma,\pi}\in\HH{\mu}$ and $\phi_{\sigma,\pi}\in\calL(\sigma)$ with  
      \begin{equation}\label{equation:WitnessNonIdentity}
         M_{\sigma,\pi}  \models_\calL ~ \phi_{\sigma,\pi} ~ \Longleftrightarrow ~ \neg(M_{\sigma,\pi}  {\models}_\calL ~ \pi(\phi_{\sigma,\pi})).
      \end{equation}
   \end{enumerate}   
  
   \item\label{item:WeakCompactFromWShrewd1.5} If $\mu$ is a cardinal with the properties listed in \eqref{item:WeakCompactFromWShrewd1}, $\sigma$ is a language in $\HH{\ooo_\calL}$ and $u$ is a renaming of $\sigma$ into a language $\tau$, then $u_*$ is the unique bijection $\varpi$ between $\calL(\sigma)$ and $\calL(\tau)$ with the property that 
    \begin{equation}\label{equation:UniquePermutation}
     M\models_\calL\phi ~ \Longleftrightarrow ~ u^*(M)\models_\calL\varpi(\phi)
    \end{equation}
    holds for every $\sigma$-structure $M$ in $\HH{\mu}$ and every $\phi\in\calL(\sigma)$. 
    
  \item\label{item:WeakCompactFromWShrewd2} If $\mu$ is a cardinal with the properties listed in \eqref{item:WeakCompactFromWShrewd1},    $\kappa>\mu$ is a $C^{(n)}$-weakly shrewd cardinal and  $\tau$ is a language with $\betrag{\SSSS_\tau}\leq\kappa$ and $\SSSS_\tau\subseteq\HH{\kappa}$, then for every cardinal $\kappa<\theta\in C^{(n+2)}$ and every $y\in\HH{\theta}$, there exists 
  \begin{itemize} 
   \item  a cardinal $\bar{\theta}\in C^{(n)}$, 
  
   \item  a cardinal $\mu<\bar{\kappa}<\min(\kappa,\bar{\theta})$,  
   
   \item     a sublanguage $\bar{\tau}$ of $\tau$ with $\betrag{\SSSS_{\bar{\tau}}}\leq\bar{\kappa}$ and $\SSSS_{\bar{\tau}}=\SSSS_\tau\cap\HH{\bar{\kappa}}$, 
    
   \item  an elementary submodel $X$ of $\HH{\bar{\theta}}$ with  $\bar{\kappa}\cup\{\bar{\kappa},\bar{\tau},\calL(\bar{\tau}),z\}\subseteq X$ and 
     
    \item   an elementary embedding $\map{j}{X}{\HH{\theta}}$       with $j\restriction\bar{\kappa}=\id_{\bar{\kappa}}$,  $j(\bar{\kappa})=\kappa$, $j(z)=z$, $j(\bar{\tau})=\tau$, $j(\calL(\bar{\tau}))=\calL(\tau)$, $j\restriction(\calL(\bar{\tau})\cap X)=\id_{\calL(\bar{\tau})\cap X}$ and $y\in\ran{j}$. 
   \end{itemize}
   
   \item\label{item:WeakCompactFromWShrewd3} If $\mu$ is a cardinal with the properties listed in \eqref{item:WeakCompactFromWShrewd1}, then every $C^{(n)}$-weakly shrewd cardinal greater than $\mu$ is a   weak compactness cardinal for $\calL$.  
   
   \item\label{item:WeakCompactFromWShrewd4} If $\mu$ is a cardinal with the properties listed in \eqref{item:WeakCompactFromWShrewd1}, then every $C^{(n)}$-weakly shrewd cardinal greater than $\mu$ is a   strict L\"owenheim--Skolem--Tarski number for $\calL$.  
   
      \item\label{item:WeakCompactFromWShrewd5} If $\mu$ is a cardinal with the properties listed in \eqref{item:WeakCompactFromWShrewd1}, then every $C^{(n)}$-weakly shrewd cardinal $\kappa>\mu$ with the property that $\nu^{{<}\ooo_\calL}\leq\kappa$ holds for all $\nu<\kappa$     is a   weak L\"owenheim--Skolem--Tarski number for $\calL$.  
 \end{enumerate}
 \end{lemma}

\begin{proof}
 \eqref{item:WeakCompactFromWShrewd1} Note that, given a language $\sigma$, an application of \eqref{item:AbstractLogic-renaming} in Definition \ref{definition:AbstractLogic} to the trivial renaming of $\sigma$ into itself ensures that the identity function on $\calL(\sigma)$ is the unique permutation $\iota$ of $\calL(\sigma)$ with the property that $$M\models_\calL \phi ~ \Longleftrightarrow ~ M\models_\calL \iota(\phi)$$ holds for every $\sigma$-structure $M$ and every $\phi\in\calL(\sigma)$. Hence, for every non-trival permuation $\pi$ of $\calL(\sigma)$, there exists a $\sigma$-structure $M_{\sigma,\pi}$ and $\phi_{\sigma,\pi}\in\calL(\sigma)$ with \eqref{equation:WitnessNonIdentity}. Using this observation, it is now easy to find a cardinal $\mu$ with the desired properties.

 \eqref{item:WeakCompactFromWShrewd1.5} Let $\mu$ be a cardinal as in \eqref{item:WeakCompactFromWShrewd1}, let $\sigma$ be a language in $\HH{\ooo_\calL}$, let $u$ be a renaming of $\sigma$ into a language $\tau$ and let $\varpi$ be a bijection between $\calL(\sigma)$ and $\calL(\tau)$ such that \eqref{equation:UniquePermutation} holds. Assume, towards a contradiction that $\varpi\neq u_*$. Then $\pi=(u_*)^{{-}1}\circ\varpi$ is a non-trivial permutation of $\calL(\sigma)$ and, since $M_{\sigma,\pi}$ is an element of $\HH{\mu}$, our assumptions on $\varpi$  ensure that $$M_{\sigma,\pi}\models_\calL\phi_{\sigma,\pi} ~ \Longleftrightarrow ~ u^*(M_{\sigma,\pi})\models_\calL\varpi(\phi_{\sigma,\pi}) ~ \Longleftrightarrow ~ M_{\sigma,\pi}\models_\calL\pi(\phi_{\sigma,\pi}),$$ contradicting the definitions of $M_{\sigma,\pi}$ and $\phi_{\sigma,\pi}$. This shows that $\varpi$ is equal to $u_*$.

 \eqref{item:WeakCompactFromWShrewd2} Fix a cardinal $\mu$ as in \eqref{item:WeakCompactFromWShrewd1}, a $C^{(n)}$-weakly shrewd cardinal $\kappa>\mu$, a cardinal $\kappa<\theta\in C^{(n+2)}$, an element $y$ of $\HH{\theta}$ and  a language  $\tau$ with $\betrag{\SSSS_\tau}\leq\kappa$ and $\SSSS_\tau\subseteq\HH{\kappa}$.  Using Lemma \ref{lemma:WcnShrewdChar}, we can find  
  a cardinal $\bar{\theta}\in C^{(n)}$, 
   a cardinal $\bar{\kappa}<\min(\kappa,\bar{\theta})$, 
   an elementary submodel $X$ of $\HH{\bar{\theta}}$ with $\bar{\kappa}+1\subseteq X$ and 
   an elementary embedding $\map{j}{X}{\HH{\theta}}$ with $j\restriction\bar{\kappa}=\id_{\bar{\kappa}}$, $j(\bar{\kappa})=\kappa$ and  $\mu,\tau,y\in\ran{j}$.  
  Then $\mu<\bar{\kappa}$ and the fact that  $\mu\in C^{(1)}$ implies that $\HH{\mu}\cup\{\HH{\mu}\}\subseteq X$ with $j\restriction\HH{\mu}=\id_{\HH{\mu}}$.  In particular, we know that $z\in X$ with $j(z)=z$. 
 Moreover, our setup ensures that there is a language $\bar{\tau}$ in $X$ with $j(\bar{\tau})=\tau$, $\SSSS_{\bar{\tau}}\subseteq\HH{\bar{\kappa}}$ and $\betrag{\SSSS_{\bar{\tau}}}\leq\bar{\kappa}$. 
  It then follows that $j\restriction\SSSS_{\bar{\tau}}=\id_{\SSSS_{\bar{\tau}}}$ and $\bar{\tau}$ is a sublanguage of $\tau$ with $\SSSS_{\bar{\tau}}=\SSSS_\tau\cap\HH{\bar{\kappa}}$. 
  Moreover, since  $\calL(\tau)$ is the unique set $\ell$ such that  $\varphi(\tau,\ell,z)$ holds and  $\calL(\bar{\tau})$ is the unique set $\bar{\ell}$ such that  $\varphi(\bar{\tau},\bar{\ell},z)$ holds, the correctness properties of $\HH{\theta}$ and $\HH{\bar{\theta}}$ imply that $\calL(\bar{\tau})\in X$ with $j(\calL(\bar{\tau}))=\calL(\tau)$.

 Now, assume, towards a contradiction, that there is $\phi_*\in\calL(\bar{\tau})\cap X$ with $j(\phi_*)\neq\phi_*$.  The correctness properties of $\HH{\bar{\theta}}$ and \eqref{item:occurrenceNumber} in Definition \ref{definition:AbstractLogic}  then allow us to find a sublanguage $\tau_0$ of $\bar{\tau}$ in $X$ with $\phi_*\in\calL(\tau_0)$ and $\betrag{\SSSS_{\tau_0}}<\ooo_\calL$. We then know that $\tau_0$ is an element of $\HH{\bar{\kappa}}\cap X$ with $j(\tau_0)=\tau_0$ and $\calL(\tau_0)$ is an element of $X$ with $j(\calL(\tau_0))=\calL(\tau_0)$.  This implies that $j(\phi_*)$ is an element of $\calL(\tau_0)$. 
 In this situation,  we can now pick a language $\tau_1$ that is an element of $\HH{\ooo_\calL}$ and a renaming $u$ of $\tau_1$ into $\tau_0$ that is an element of $\HH{\bar{\kappa}}\cap X$.  It follows that $j(u)=u$. 
 Let  $u^*$ denote the canonical bijection between $\Str_{\tau_1}$ and $\Str_{\tau_0}$ induced by $u$. 
  Since $\tau_0$ and $u$ are both elements of $\HH{\bar{\kappa}}\cap X$, it follows that  $u^*(M)\in\HH{\bar{\kappa}}\cap X$ holds for every $\tau_1$-structure $M$ in $\HH{\mu}$. 
 In particular,  we know that   $j(u^*(M))=u^*(M)$ holds for every  $\tau_1$-structure  $M$ in $\HH{\mu}$.  
 Let $\map{u_*}{\calL(\tau_1)}{\calL(\tau_0)}$ denote the unique bijection given by \eqref{item:AbstractLogic-renaming} in Definition \ref{definition:AbstractLogic}.

 \begin{claim*}
  The map $u_*$ is an element of $X$.  
 \end{claim*}
 
 \begin{proof}[Proof of the Claim]
   First,   since $\theta$ is an element of $C^{(n+2)}$, if follows  that $u_*$ is an element of $\HH{\theta}$. The elementarity of $j$ and the correctness properties of $\HH{\bar{\theta}}$ then imply that there is a bijection $\map{\varpi}{\calL(\tau_1)}{\calL(\tau_0)}$ in $X$ with the property that $$M\models_\calL\phi ~ \Longleftrightarrow ~ u^*(M)\models_\calL\varpi(\phi)$$ holds for all $M\in\Str_{\tau_1}\cap X$ and all $\phi\in\calL(\tau_1)$. 
   Since $\HH{\mu}\subseteq X$, we can now apply \eqref{item:WeakCompactFromWShrewd1.5} to conclude that $\varpi=u_*$.  
 \end{proof}

  Since our setup ensures that $\calL(\tau_1)\subseteq X$, we now know that $\calL(\tau_0)=u_*[\calL(\tau_1)]\subseteq X$ and the fact that  $j(\calL(\tau_0))=\calL(\tau_0)$ implies that $$\map{j\restriction\calL(\tau_0)}{\calL(\tau_0)}{\calL(\tau_0)}$$ is a bijection. 
     If we now define $$\map{\pi ~ = ~ (u_*)^{{-}1}\circ(j\restriction\calL(\tau_0))\circ u_*}{\calL(\tau_1)}{\calL(\tau_1)},$$ then we have  $$\pi((u_*)^{{-}1}(\phi_*)) ~ = ~ (u_*)^{{-}1}(j(\phi_*)) ~ \neq ~ (u_*)^{{-}1}(\phi_*)$$ and this shows that $\pi$ is a non-trivial permutation of $\calL(\tau_1)$. 
  Since $M_{\tau_1,\pi}\in\HH{\mu}\subseteq\HH{\bar{\kappa}}\cap X$, we know that $j(M_{\tau_1,\pi})=M_{\tau_1,\pi}$. 
  Moreover, earlier computations show that $u^*(M_{\tau_1,\pi})$ is an element of $X$ with $j(u^*(M_{\tau_1,\pi}))=u^*(M_{\tau_1,\pi})$. 
    Since $u_*(\phi_{\tau_1,\pi})\in\calL(\tau_0)\subseteq X$, we can now conclude that 
  \begin{equation*}
   \begin{split}
     M_{\tau_1,\pi}\models_\calL ~ \phi_{\tau_1,\pi} ~ & \Longleftrightarrow ~ u^*(M_{\tau_1,\pi})\models_\calL  u_*(\phi_{\tau_1,\pi}) \\
      &  \Longleftrightarrow ~ X\models\psi(u^*(M_{\tau_1,\pi}),u_*(\phi_{\tau_1,\pi}),z)   \\ 
    &  \Longleftrightarrow ~ \HH{\theta}\models\psi(u^*(M_{\tau_1,\pi}),(j\circ u_*)(\phi_{\tau_1,\pi}),z) \\
     & \Longleftrightarrow ~ u^*(M_{\tau_1,\pi})\models_\calL (u_*\circ\pi)(\phi_{\tau_1,\pi}) \\ 
     & \Longleftrightarrow ~ M_{\tau_1,\pi}\models_\calL ~ \pi(\phi_{\tau_1,\pi}),  
   \end{split} 
  \end{equation*}
  where the first and fifth equivalence is given by \eqref{item:AbstractLogic-renaming} in  Definition \ref{definition:AbstractLogic}, the second equivalence is given by the correctness properties of $\HH{\bar{\theta}}$, the third equivalence is given by the elementarity of $j$ and the fourth equivalence is given by the correctness properties of $\HH{\theta}$. 
   The above equivalences now contradict the defining properties of $M_{\tau_1,\pi}$ and $\phi_{\tau_1,\pi}$. 
  This shows that $j\restriction(\calL(\bar{\tau})\cap X)=\id_{\calL(\bar{\tau})\cap X}$.

  \eqref{item:WeakCompactFromWShrewd3} Let $\mu$ be a cardinal as  in \eqref{item:WeakCompactFromWShrewd1}, let $\kappa>\mu$ be a $C^{(n)}$-weakly shrewd cardinal and let  $U$ be a ${<}\kappa$-satisfiable $\calL$-theory  of cardinality $\kappa$. 
   Since $\ooo_\calL<\kappa$, we can find a language $\sigma$  with $\betrag{\SSSS_\sigma}\leq\kappa$ and $U\subseteq\calL(\sigma)$. 
   Then there is a renaming $u$ of $\sigma$ into a language $\tau$ with $\SSSS_\tau\subseteq\HH{\kappa}$.  Set $T=u_*[U]$. Then \eqref{item:AbstractLogic-renaming} in Definition \ref{definition:AbstractLogic} ensures that $T$ is also ${<}\kappa$-satisfiable. 
  Now, pick a cardinal $\kappa<\theta\in C^{(n+2)}$ with $T\in\HH{\theta}$ and   use \eqref{item:WeakCompactFromWShrewd2} to find 
  a cardinal $\bar{\theta}\in C^{(n)}$,  a cardinal $\mu<\bar{\kappa}<\min(\kappa,\bar{\theta})$,    a sublanguage $\bar{\tau}$ of $\tau$ with $\betrag{\SSSS_{\bar{\tau}}}\leq\bar{\kappa}$ and $\SSSS_{\bar{\tau}}=\SSSS_\tau\cap\HH{\bar{\kappa}}$,      an elementary submodel $X$ of $\HH{\bar{\theta}}$ with  $\bar{\kappa}\cup\{\bar{\kappa},\bar{\tau},\calL(\bar{\tau}),z\}\subseteq X$ and   an elementary embedding $\map{j}{X}{\HH{\theta}}$       with $j\restriction\bar{\kappa}=\id_{\bar{\kappa}}$,  $j(\bar{\kappa})=\kappa$, $j(z)=z$, $j(\bar{\tau})=\tau$, $j(\calL(\bar{\tau}))=\calL(\tau)$, $j\restriction(\calL(\bar{\tau})\cap X)=\id_{\calL(\bar{\tau})\cap X}$ and $T\in\ran{j}$. 
  Pick $\bar{T}\in X$ with $j(\bar{T})=T$. Then $\bar{T}$ is a subset of $\calL(\bar{\tau})$ of cardinality $\bar{\kappa}$ and the fact that $\bar{\kappa}\subseteq X$ ensures that $\bar{T}\subseteq X$. 
  Moreover, since $j\restriction(\calL(\bar{\tau})\cap X)=\id_{\calL(\bar{\tau})\cap X}$, it follows that $\bar{T}=j[\bar{T}]\subseteq T$ and hence $\bar{T}$ is satisfiable. This shows that there is a $\bar{\tau}$-structure $M$ with the property that $\psi(M,\phi,z)$ holds for all $\phi$ in $\bar{T}$. The correctness properties of $\HH{\bar{\theta}}$ now ensure that this statement holds in $X$ and hence we know that, in $\HH{\theta}$, there is a $\tau$-structure $N$ with the property that $\psi(N,\phi,z)$ holds for all $\phi$ in $T$. Using the correctness properties of $\HH{\theta}$, we can now conclude that $T$ is satisfiable and, by   \eqref{item:AbstractLogic-renaming} in  Definition \ref{definition:AbstractLogic}, we also know that $U$ is  satisfiable.

 \eqref{item:WeakCompactFromWShrewd4}  Let $\mu$ be a cardinal as  in \eqref{item:WeakCompactFromWShrewd1}, let  $\kappa>\mu$ be a $C^{(n)}$-weakly shrewd cardinal,  let  $\sigma$ be a language with the property that $\betrag{\SSSS_\sigma}<\kappa$, let $I$ be a $\sigma$-structure  with $\betrag{\dom{I}}=\kappa$ and let  $\phi$ be an element of $\calL(\sigma)$ with $I\models_\calL\phi$. Let $u$ be a renaming of $\sigma$ into a language $\tau\in\HH{\kappa}$. Set $K=u^*(I)$ and $\phi_1=u_*(\phi)$. Then $K\models_\calL\phi_1$ and there is an isomorphism $\pi$ between $K$ and a $\tau$-structure $N$ with $\dom{N}=\kappa$ and $N\models_\calL\phi_1$. Fix a cardinal $\kappa<\theta\in C^{(n+2)}$. Then $K$ is an element of $\HH{\theta}$ and we can use \eqref{item:WeakCompactFromWShrewd2} to find a cardinal $\bar{\theta}\in C^{(n)}$, a cardinal $\mu<\bar{\kappa}<\min(\kappa,\bar{\theta})$,   a sublanguage $\bar{\tau}$ of $\tau$ with $\betrag{\SSSS_{\bar{\tau}}}\leq\bar{\kappa}$ and $\SSSS_{\bar{\tau}}=\SSSS_\tau\cap\HH{\bar{\kappa}}$,      an elementary submodel $X$ of $\HH{\bar{\theta}}$ with  $\bar{\kappa}\cup\{\bar{\kappa},\bar{\tau},\calL(\bar{\tau}),z\}\subseteq X$ and   an elementary embedding $\map{j}{X}{\HH{\theta}}$       with $j\restriction\bar{\kappa}=\id_{\bar{\kappa}}$,  $j(\bar{\kappa})=\kappa$, $j(z)=z$, $j(\bar{\tau})=\tau$, $j(\calL(\bar{\tau}))=\calL(\tau)$, $j\restriction(\calL(\bar{\tau})\cap X)=\id_{\calL(\bar{\tau})\cap X}$ and $N,\phi_1\in\ran{j}$. 
 The fact that $\tau$ is an element of $\HH{\kappa}$ then implies that $\tau=\bar{\tau}\in\HH{\bar{\kappa}}$, $\SSSS_\tau\subseteq X$ and $j\restriction\SSSS_\tau=\id_{\SSSS_\tau}$. Pick $M\in X$ with $j(M)=N$. 
 Then $M$ is a $\tau$-structure with $\dom{M}=\bar{\kappa}$ and it follows that $M$ is a   substructure of $N$. 
 Next, pick $\phi_0$ in $X$ with $j(\phi_0)=\phi_1$. Then $\phi_0\in\calL(\tau)$ and the fact that $j\restriction(\calL(\tau)\cap X)=\id_{\calL(\tau)\cap X}$ implies that $\phi_0=\phi_1\in X$. Moreover, using the elementarity of $j$ and the correctness properties of $\HH{\theta}$ and $\HH{\bar{\theta}}$, we know that $M\models_\calL\phi_1$ holds. Then $\pi^{{-}1}\restriction\bar{\kappa}$ is an isomorphism between $M$ and a substructure $J$ of $K$, and we can apply \eqref{item:AbstractLogic-Iso} in  Definition \ref{definition:AbstractLogic} to show that $J\models_\calL\phi_1$. Finally, we can use \eqref{item:AbstractLogic-Iso} in  Definition \ref{definition:AbstractLogic} to conclude  that $H=(u^*)^{{-}1}(J)$ is a substructure of $I$ with $\betrag{\dom{H}}<\kappa$ and $H\models_\calL\phi$.

 \eqref{item:WeakCompactFromWShrewd5} Given a cardinal $\mu$ as  in \eqref{item:WeakCompactFromWShrewd1}, a $C^{(n)}$-weakly shrewd cardinal $\kappa>\mu$ satisfying $\nu^{{<}\ooo_\calL}\leq\kappa$  for all $\nu<\kappa$, a language  $\sigma$  with  $\betrag{\SSSS_\sigma}<\kappa$ and a $\sigma$-structure $I$  with $\betrag{\dom{I}}=\kappa$, we  proceed as in \eqref{item:WeakCompactFromWShrewd4} to first find a renaming $u$ of $\sigma$ into a language $\tau\in\HH{\kappa}$ and an isomorphism $\pi$ between $u^*(I)$ and a $\tau$-structure $N$ with $\dom{N}=\kappa$. As above, fix a cardinal $\kappa<\theta\in C^{(n+2)}$ and use  \eqref{item:WeakCompactFromWShrewd2} to find a cardinal $\bar{\theta}\in C^{(n)}$, a cardinal $\mu<\bar{\kappa}<\min(\kappa,\bar{\theta})$ with $\tau\in\HH{\bar{\kappa}}$,      an elementary submodel $X$ of $\HH{\bar{\theta}}$ with  $\bar{\kappa}\cup\{\bar{\kappa},\tau,\calL(\tau),z\}\subseteq X$  and   an elementary embedding $\map{j}{X}{\HH{\theta}}$       with $j\restriction\bar{\kappa}=\id_{\bar{\kappa}}$,  $j(\bar{\kappa})=\kappa$,  $j(z)=z$, $j(\tau)=\tau$, $j(\calL(\tau))=\calL(\tau)$, $j\restriction(\calL(\tau)\cap X)=\id_{\calL(\tau)\cap X}$ and $N\in\ran{j}$. This setup then ensures that $\nu^{{<}\ooo_\calL}\leq\bar{\kappa}$ holds for all $\nu<\bar{\kappa}$.

 \begin{claim*}
  $\calL(\tau)\subseteq X$. 
 \end{claim*}
 
 \begin{proof}[Proof of the Claim]
   Fix $\phi\in\calL(\tau)$. By   \eqref{item:occurrenceNumber} in  Definition \ref{definition:AbstractLogic}, there is a sublanguage $\tau_0$ of $\tau$ with $\betrag{\SSSS_{\tau_0}}<\ooo_\calL$ and $\phi\in\calL(\tau_0)$. Since $\betrag{\SSSS_\tau}^{{<}\ooo_\calL}\leq\bar{\kappa}\subseteq X$, it follows that $\tau_0$ is an element of $\HH{\bar{\kappa}}\cap X$ and this implies that $j(\tau_0)=\tau_0$ and $j(\calL(\tau_0))=\calL(\tau_0)$. 
   Now, we can find a language $\tau_1\in\HH{\ooo_\calL}$ and a renaming $v$ of $\tau_1$ into $\tau_0$ that is an element of $X$.   Our setup then ensures that $\calL(\tau_1)\in\HH{\mu}\subseteq X$ and $j(\calL(\tau_1))=\calL(\tau_1)$. 
   Since \eqref{item:WeakCompactFromWShrewd1.5} shows that $v_*$ is the unique bijection between $\calL(\tau_1)$ and $\calL(\tau_0)$ with the property that \eqref{equation:UniquePermutation} holds   for all $\tau_1$-structures $M$ in $\HH{\mu}$ and every $\phi\in\calL(\tau_1)$, the fact that all parameters of this existence statement are elements of $X$ that are fixed by $j$ ensures that $v_*$ is an element of $X$. 
   In particular, since $\calL(\tau_1)\subseteq\HH{\mu}\subseteq X$, we know that   $\calL(\tau_0)\subseteq X$ and hence $\phi$ is an element of $X$.  
  \end{proof}

  Now, pick $M\in X$ with $j(M)=N$. Then $M$ is a $\tau$-structure with $\dom{M}=\bar{\kappa}$. Moreover, since $j\restriction\SSSS_\tau=\id_{\SSSS_\tau}$, we know that $M$ is a substructure of $N$. Fix $\phi\in\calL(\tau)$. Then $\phi\in X$ with $j(\phi)=\phi$ and we can conclude that $$M\models_\calL\phi ~  \Longleftrightarrow ~ X\models\psi(M,\phi,z) ~  \Longleftrightarrow ~ \HH{\theta}\models\psi(N,\phi,z) ~ \Longleftrightarrow ~ N\models_\calL\phi.$$ This shows that $M$ is an $\calL$-elementary submodel of $N$ and, by arguing as in the proof of \eqref{item:WeakCompactFromWShrewd4}, this shows that there is an $\calL$-elementary submodel $H$ of $I$ with $\betrag{\dom{H}}<\kappa$.  
\end{proof}

By combining the above lemma with Theorems \ref{theorem:OrdSubtle1} and \ref{theorem:CharOrdFaint}, we obtain the forward implications in Theorems \ref{MAIN:EssentiallySubtle}, \ref{MAIN:EssentiallyFaint} and \ref{MAIN:StrictLST}:

\begin{corollary}\label{corollary:CompactnessCardinalsFromSubtle}
 \begin{enumerate}
   \item  If $\Ord$ is essentially faint and $\calL$ is an abstract logic, then there is a proper class of cardinals that are both weak  compactness cardinals for $\calL$  and  strict L\"owenheim--Skolem--Tarski numbers for $\calL$. 
   
   
   \item If $\Ord$ is essentially subtle and $\calL$ is an abstract logic, then the class of cardinals that are both weak compactness cardinals for $\calL$  and  weak L\"owenheim--Skolem--Tarski numbers for $\calL$ is stationary.   
 \end{enumerate}
\end{corollary}

\begin{proof}
 Given a natural number $n>0$, assume that $\calL$ is an abstract logic with the property that the function $\calL$ and the relation $\models_\calL$ are definable by $\Sigma_n$-formulas with parameters. Then \eqref{item:WeakCompactFromWShrewd3} and \eqref{item:WeakCompactFromWShrewd4} of Lemma \ref{lemma:CnWeaklyShrewd} show that all sufficiently large $C^{(n)}$-weakly shrewd cardinals are both weak compactness cardinals for $\calL$  and  strict L\"owenheim--Skolem--Tarski numbers for $\calL$. 
 Theorem \ref{theorem:CharOrdFaint} then shows that the assumption that $\Ord$ is essentially faint implies that there is a proper class of cardinals that are both weak compactness cardinals for $\calL$  and  strict L\"owenheim--Skolem--Tarski numbers for $\calL$.

 Now, let $\calL$ be an abstract logic and let $C$ be a closed unbounded class of ordinals. Then there is a natural number $n>0$ such that both $\calL$ and $C$ are definable by $\Sigma_n$-formulas with parameters. Since every sufficiently large cardinal in $C^{(n)}$ is an element of $C$,  we can combine the above computations with Lemma \ref{lemma:separatePropertiesCn} to conclude that every sufficiently large $C^{(n)}$-strongly unfoldable cardinal is an element of $C$ that is both a weak compactness cardinal for $\calL$  and a weak L\"owenheim--Skolem--Tarski number  for $\calL$. In particular,   Theorem \ref{theorem:OrdSubtle1} shows that the assumption that $\Ord$ is essentially subtle implies that every closed unbounded class of ordinals contains a cardinal that is both a weak compactness cardinal for $\calL$  and a weak L\"owenheim--Skolem--Tarski number for $\calL$. 
\end{proof}

 By combining the above corollary with arguments contained in \cite{bdgm}, we can also use the obtained results to separate the principle  \anf{\emph{$\Ord$ is  subtle}} from the principle  \anf{\emph{$\Ord$ is essentially subtle}}:

 \begin{corollary}\label{corollary:SublteNotEssentially}
  If the theory $$\ZFC ~ + ~ \anf{\textit{$\Ord$ is subtle}}$$ is consistent, then this theory does not prove that $\Ord$ is essentially subtle. 
 \end{corollary}
 
 \begin{proof}
  The proof of {\cite[Theorem 5.6]{bdgm}} shows that, if the above theory is consistent, then it is consistent with the assumption that there is no weak compactness cardinal for second-order logic. 
  An application of Corollary \ref{corollary:CompactnessCardinalsFromSubtle} then shows that this theory does not prove that $\Ord$ is essentially subtle. 
 \end{proof}

 We now work towards proving the backward implications in Theorems \ref{MAIN:EssentiallySubtle}, \ref{MAIN:EssentiallyFaint} and \ref{MAIN:StrictLST}. 
  Given a natural number $n>0$ and a cardinal $\kappa$ that is not $C^{(n)}$-weakly shrewd, 
   we let $\theta^n_\kappa$ denote the minimal element $\theta$ of $C^{(n)}$ greater than $\kappa$ with the property that there exists $z\in \HH{\theta}$ such that for all $\bar{\theta}\in C^{(n)}$, all cardinals $\bar{\kappa}<\min(\kappa,\bar{\theta})$ and all  elementary submodels $X$ of $\HH{\bar{\theta}}$ with $\bar{\kappa}+1\subseteq X$, there is no elementary embedding $\map{j}{X}{\HH{\theta}}$ with $j\restriction\bar{\kappa}=\id_{\bar{\kappa}}$, $j(\bar{\kappa})=\kappa$ and $z\in\ran{j}$. 
 In addition, given a natural number $n>0$ and a cardinal $\kappa$ that is not $C^{(n)}$-weakly shrewd,  we let $Z^n_\kappa$ denote the non-empty set of all elements $z$ of $\HH{\theta^n_\kappa}$ with the property that for all $\bar{\theta}\in C^{(n)}$, all cardinals $\bar{\kappa}<\min(\kappa,\bar{\theta})$ and all  elementary submodels $X$ of $\HH{\bar{\theta}}$ with $\bar{\kappa}+1\subseteq X$, there is no elementary embedding $\map{j}{X}{\HH{\theta^n_\kappa}}$ with $j\restriction\bar{\kappa}=\id_{\bar{\kappa}}$, $j(\bar{\kappa})=\kappa$ and $z\in\ran{j}$.  
 In the following,  let $\sigma$ denote the language consisting of a binary relation symbol $E$ and four constant symbols $d$, $e$,  $f$ and $g$. 
 Given a natural number $m$, a natural number $n>0$ and a cardinal $\rho\in C^{(1)}$ of uncountable cofinality, we  define $\Ce^{m,n}_\rho$ to be the class of all $\sigma$-structures $M$ 
 with the property that there exists a (necessarily unique) cardinal $\rho<\kappa_M\in\Lim(C^{(m)})$ such that the following statements hold:  
  \begin{itemize}
   
   \item There exists a cardinal $\vartheta>\kappa_M$ with the property that $\vartheta>\theta^n_\mu$ holds for every cardinal $\mu$ in $C^{(m)}\cap(\rho,\kappa_M]$ that is not $C^{(n)}$-weakly shrewd and $\dom{M}$ is an elementary submodel of $\HH{\vartheta}$ of cardinality $\kappa_M$ with $\kappa_M+1\subseteq \dom{M}$. 
   
      \item $E^M={\in}\restriction(\dom{M}\times\dom{M})$, $d^M=\kappa_M$ and $e^M=\rho$.  
      
   \item $f^M$ and $g^M$ are functions whose domains are the set of all cardinals in  $C^{(m)}\cap(\rho,\kappa_M]$ that are not $C^{(n)}$-weakly shrewd.  
   
   \item If $\mu\in\dom{f^M}$, then $f^M(\mu)$ is an elementary submodel of $\HH{\theta^n_\mu}$ of cardinality $\mu$ with $\mu+1\subseteq f^M(\mu)$ and $g^M(\mu)\in Z^n_\mu\cap f^M(\mu)$. 
  \end{itemize}

 The following lemma shows how the above class of structures can be used to derive the existence of $C^{(n)}$-weakly shrewd cardinals from strong structural properties of abstract logics:

\begin{lemma}\label{lemma:CardinalFromLogic}
  Given a natural number $m$, a natural number $n>0$ and  a   cardinal $\rho\in C^{(1)}$ of uncountable cofinality, assume that $\calL$ is an abstract logic with the property that  the  following statements hold whenever $\tau$ is a language extending $\sigma$:  
   \begin{enumerate}
    \item\label{item:CardinalFromLogic1} The set $\calL(\tau)$ contains all first-order sentences of $\tau$, a distinguished sentence $\phi_*$ and  distinguished sentences $\phi_{{<}\xi}$ and $\phi_{{\geq}\xi}$ for every ordinal $\xi\leq\rho$.  
    
    \item Given $M\in\Str_\tau$ and  a first-order sentence $\phi\in\calL(\tau)$,      $M\models_{\calL}\phi$ holds if and only if $\phi$ holds in $M$ (in the sense of first-order structures). 
    
    \item Given $M\in\Str_\tau$,      $M\models_{\calL}\phi_*$ holds if and only if $M\restriction\sigma$ is isomorphic to a structure in $\Ce^{m,n}_\rho$. 
    
    \item Given $\xi\leq\rho$ and $M\in\Str_\tau$,      $M\models_{\calL}\phi_{{<}\xi}$ holds if and only if $\langle\dom{M},E^M\rangle$ is a well-order of order-type less than $\xi$. 
    
     \item\label{item:CardinalFromLogic5} Given $\xi\leq\rho$ and $M\in\Str_\tau$,      $M\models_{\calL}\phi_{{\geq}\xi}$ holds if and only if $\langle\dom{M},E^M\rangle$ is a well-order of order-type at least $\xi$. 
  \end{enumerate}
  
  If $\kappa\in\Lim(C^{(m)})$ is either a weak compactness cardinal for $\calL$ or a strict L\"owenheim--Skolem--Tarski number for $\calL$, then $\kappa>\rho$ and the set $C^{(m)}\cap(\rho,\kappa]$ contains a   $C^{(n)}$-weakly shrewd cardinal $\mu$ with the property that $\delta^\eta<\mu$ holds for all $\delta<\mu$ and $\eta\leq\rho$ with $\delta^\eta<\kappa$. 
\end{lemma}

\begin{proof}
 First, assume towards a contradiction, that $\kappa\leq\rho$. Define $$T ~ = ~ \{\phi_{{<}\kappa}\} ~ \cup ~ \Set{\phi_{{\geq}\xi}}{\xi<\kappa} ~ \subseteq ~ \calL(\sigma).$$ Then $T$ is an unsatisfiable $\calL$-theory that is ${<}\kappa$-satisfiable. This shows that $\kappa$ is not a weak compactness cardinal for $\calL$. 
 If we now pick a $\sigma$-structure $N$ with $\dom{N}=\kappa$ and $E^N={\in}\restriction(\kappa\times\kappa)$, then $N\models_\calL\phi_{{\geq}\kappa}$ holds and $M\models_\calL\phi_{{\geq}\kappa}$ fails for every substructure $M$ of $N$ with $\betrag{\dom{M}}<\kappa$. This shows that $\kappa$ is also not a strict L\"owenheim--Skolem--Tarski number for $\calL$, contradicting our initial assumptions. Therefore, we know that  $\kappa>\rho$ and we can pick  $M\in\Ce^{m,n}_\rho$ with $\kappa_M=\kappa$.

  \begin{claim*}
   There exist $\sigma$-structures $H$ and $K$ in $\Ce^{m,n}_\rho$ with $\kappa\in\{\kappa_H,\kappa_K\}$  
    and the property that there is an elementary embedding $\map{j}{\dom{H}}{\dom{K}}$ satisfying $j(\kappa_H)=\kappa_K$, $j(\rho)=\rho$, $j\restriction(\kappa_H+1)\neq\id_{\kappa_H+1}$, $j(f^H)=f^K$ and $j(g^H)=g^K$.   
  \end{claim*}
  
  \begin{proof}[Proof of the Claim]
    First, assume that $\kappa$ is a weak compactness cardinal for $\calL$.  
      Let $\tau$ denote the language that extends $\sigma$ by a constant symbol $c$ and a constant symbol  $c_x$ for every $x\in\dom{M}$.
    Next, let $T$ denote the $\calL(\tau)$-theory consisting of the union of the first-order elementary diagram of $M$ using the constant symbols $c_x$ and the set $$\{\phi_*\} ~ \cup ~ \{c\mathbin{E}d\} ~ \cup ~ \Set{c\neq c_\alpha}{\alpha<\kappa}.$$
  This theory has cardinality $\kappa$ and $\tau$-expansions of the structure $M$ witness that it is ${<}\kappa$-satisfiable. 
  Hence, our assumption implies that  $T$ is satisfiable. Our setup and \eqref{item:AbstractLogic-Iso} in Definition \ref{definition:AbstractLogic} now ensure that $T$ has a model whose $\sigma$-reduct $N$ is an element of $\Ce^{m,n}_\rho$. 
  We can now find an elementary embedding $\map{j}{\dom{M}}{\dom{N}}$ with $j(\kappa)=\kappa_N\geq\kappa$, $j(\rho)=\rho$, $c^N\in\kappa_N\setminus j[\kappa]$, $j(f^M)=f^N$ and $j(g^M)=g^N$.   In particular, we know that $j\restriction(\kappa+1)\neq\id_{\kappa+1}$.

    Now, assume that $\kappa$ is a strict L\"owenheim--Skolem--Tarski number for $\calL$. Let  $\tau$ be a language extending $\sigma$ that adds functions symbols for first-order Skolem functions to $\sigma$ and let $N$ be a $\tau$-structure such that $N\restriction\sigma=M$ and $N$ interprets all new function symbols as Skolem functions. 
  Then $N\models_\calL\phi_*$ and $\betrag{\dom{N}}=\kappa$. Therefore, our assumptions yield a substructure $K$ of $N$ with $K\models_\calL\phi_*$ and $\betrag{\dom{K}}<\kappa$. We can now find a $\sigma$-structure $H$ in $\Ce^{m,n}_\rho$ that is isomorphic to $K\restriction\sigma$. 
  Our setup now ensures that there is an elementary embedding $\map{j}{\dom{H}}{\dom{M}}$ with $j(\kappa_H)=\kappa$, $j(\rho)=\rho$, $j(f^H)=f^M$ and $j(g^H)=g^M$. 
 Moreover, since $\kappa_H=\betrag{\dom{H}}<\kappa$, we know that $j\restriction(\kappa_H+1)\neq\id_{\kappa_H+1}$.  
  \end{proof}

  Fix $\sigma$-structures $H$ and $K$ in $\Ce^{m,n}_\rho$ and  an elementary embedding $\map{j}{\dom{H}}{\dom{K}}$ as in the above claim. 
  Since $j\restriction(\kappa_H+1)\neq\id_{\kappa_H+1}$, we can find an  ordinal $\mu\leq\kappa_H$ with $j\restriction\mu=\id_\mu$ and $j(\mu)>\mu$. The fact that $\dom{H}$ is an elementary submodel of  $\HH{\vartheta}$ for some cardinal $\vartheta>\kappa$ then implies that $\mu$ is a regular cardinal.

  \begin{claim*}
   $\mu>\rho$. 
  \end{claim*}
  
  \begin{proof}[Proof of the Claim]
   Since $\rho$ is an element of $C^{(1)}\cap\kappa_H$, we know that $\VV_\rho\cup\{\VV_\rho\}\subseteq\dom{H}$. Moreover, since $j(\rho)=\rho$, it follows that   $j(\VV_\rho)=\VV_\rho$ and $\map{j\restriction\VV_\rho}{\VV_\rho}{\VV_\rho}$. In this situation, the fact that  $\rho$ has uncountable cofinality now allows us to use the  \emph{Kunen Inconsistency} to conclude that $j\restriction\VV_\rho=\id_{\VV_\rho}$.  
  \end{proof}

  \begin{claim*}
   $\mu\in C^{(m)}$. 
  \end{claim*}
  
  \begin{proof}[Proof of the Claim]
   Assume, towards a contradiction, that $\mu\notin C^{(m)}$. Set $\xi=\min(C^{(m)}\setminus\mu)<\kappa_H$. Then there is an ordinal $\eta<\mu$ with the property that  the ordinal $\xi$ is definable in $\dom{H}$ by a formula using the parameters $\kappa_H$ and $\eta$ and the same formula with the parameters $\kappa_K$ and $\eta$ defines $\xi$ in $\dom{K}$. It  follows that $j(\xi)=\xi>\mu$. 
   Since $C^{(0)}$ is the class of all ordinals, we now know $m>0$ and this implies that  $\kappa_H\in C^{(1)}$ and $\VV_{\kappa_H}\subseteq\dom{H}$. The fact that $j(\VV_{\xi+2})=\VV_{\xi+2}$ then implies that $\map{j\restriction\VV_{\xi+2}}{\VV_{\xi+2}}{\VV_{\xi+2}}$ is a non-trivial elementary embedding, contradicting the  \emph{Kunen Inconsistency}.   
  \end{proof}

  \begin{claim*}
   If $\delta<\mu$ and $\eta\leq\rho$ with $\delta^\eta<\kappa$, then $\delta^\eta<\mu$. 
  \end{claim*}
  
  \begin{proof}[Proof of the Claim]
   Assume, towards a contradiction, that there are cardinals $\delta<\mu$ and $\eta\leq\rho$ with $\mu\leq\delta^\eta<\kappa$. Then $\dom{H}$ contains an injection $\map{\iota}{\mu}{{}^\eta\delta}$. Moreover, since $\delta,\eta<\mu$, $j(\kappa_H)=\kappa_K$ and $\kappa\in\{\kappa_H,\kappa_K\}$, the elementarity of $j$ ensures that $\delta^\eta<\kappa_H$ and hence we know that ${}^\eta\delta\subseteq\dom{H}$.  Set $s=j(\iota)(\mu)\in{}^\eta\delta$. Since $j(s)=s$, we can now find  an $\alpha<\mu$ with $\iota(\alpha)=s$. But then $j(\iota)(\alpha)=j(\iota(\alpha))=j(s)=j(\iota)(\mu)$, contradicting the injectivity of  $\iota$.  
  \end{proof}

  Since $\kappa_H\leq j(\kappa_H)=\kappa_K$ and $\kappa\in\{\kappa_H,\kappa_K\}$, we know that $\mu\leq\kappa_H\leq\kappa$ and the   above computations show that the cardinal $\mu$ is an element of the set $C^{(m)}\cap(\rho,\kappa]$. Assume, towards a contradiction, that $\mu$ is not a $C^{(n)}$-weakly shrewd cardinal. Then $\mu\in\dom{f^H}$.    Set $\nu=j(\mu)>\mu$. Then $\nu$ is a regular cardinal in the interval $(\rho,\kappa_K]$ and elementarity implies that $\nu\in\dom{f^K}$. 
   Set $X=f^H(\mu)\subseteq\dom{H}$ and $z=g^N(\nu)=j(g^M(\mu))\in Z^n_\nu$.  
Then $\theta^n_\mu$ is an element of $C^{(n)}$, $\mu<\min(\nu,\theta^n_\mu)$, $X$ is an elementary submodel  of $\HH{\theta^n_\mu}$ with $\mu+1\subseteq X$ and  $\map{i=j\restriction X}{X}{\HH{\theta^n_\nu}}$ is an elementary embedding with  $i\restriction\mu=\id_\mu$, $i(\mu)=\nu$ and $z\in\ran{i}$. 
  But, this contradicts the fact that $z$ is an element of $Z^n_\nu$.  
  This yields the conclusion of the lemma.   
\end{proof}

 We can now also combine the above lemma with Theorems \ref{theorem:OrdSubtle1} and \ref{theorem:CharOrdFaint} to derive the backward implications of Theorems \ref{MAIN:EssentiallySubtle}, \ref{MAIN:EssentiallyFaint} and \ref{MAIN:StrictLST}. This completes the proofs of these results.

\begin{corollary}
 \begin{enumerate}
  \item\label{item:SubtleOrdFromLogic1}  If every abstract logic has either a weak compactness cardinal or a strict L\"owenheim--Skolem--Tarski number, then $\Ord$ is essentially faint. 
  
  \item\label{item:SubtleOrdFromLogic2}  If every abstract logic has,  for every natural number $n$, either a weak compactness cardinal that is an element of $C^{(n)}$ or a strict L\"owenheim--Skolem--Tarski number that is an element of $C^{(n)}$, then $\Ord$ is essentially subtle. 

 \end{enumerate}
\end{corollary}

\begin{proof}
  \eqref{item:SubtleOrdFromLogic1} Fix a natural number  $n>0$ and   a cardinal $\rho$ in $C^{(1)}$ of uncountable cofinality. 
  Set $m=0$. 
  Then there exists an abstract logic $\calL$ that satisfies the statements \eqref{item:CardinalFromLogic1}--\eqref{item:CardinalFromLogic5} listed in Lemma \ref{lemma:CardinalFromLogic} for $m$, $n$ and $\rho$. 
  The fact that $C^{(m)}$ is the class of all ordinals now allows us to apply the lemma to find a $C^{(n)}$-weakly shrewd cardinal greater than $\rho$. 
  
  These computations show that for every natural number $n>0$, there is a proper class of $C^{(n)}$-weakly shrewd cardinals. In this situation,  Theorem \ref{theorem:CharOrdFaint} allows us to conclude that $\Ord$ is essentially faint.

  \eqref{item:SubtleOrdFromLogic2} Given a natural number  $n>0$   and a cardinal $\rho\in C^{(1)}$ of uncountable cofinality,  there exists an abstract logic $\calL$ that satisfies the statements \eqref{item:CardinalFromLogic1}--\eqref{item:CardinalFromLogic5} listed in Lemma \ref{lemma:CardinalFromLogic} for $m=n+1$, $n$ and $\rho$. Since $C^{(n+2)}\subseteq\Lim(C^{(n+1)})$, our assumptions allows us to use the lemma to find a $C^{(n)}$-weakly shrewd cardinal greater than $\rho$ that is an element of $C^{(n+1)}$. We can then combine Theorem \ref{theorem:OrdSubtle1} with Lemma \ref{lemma:separatePropertiesCn} to conclude that $\Ord$ is essentially subtle. 
\end{proof}

\begin{remark}
 In \cite{FuSa22}, Fuchino and Sakai consider the possibility of characterizing the existence of weak compactness cardinals for second-order logic $\calL^2$ through the existence of certain large cardinals, analogous to  Magidor's characterization of  the existence of strong compactness cardinals for  $\calL^2$  through the existence of extendible cardinals  in \cite{MR295904}.  
  For this purpose, they introduce the concept of \emph{weak extendibility} as a natural candidate for a large cardinal property  providing such a characterization. Here, a cardinal $\kappa$ is defined to be \emph{weakly extendible} if $2^{{<}\kappa}=\kappa$ holds and for every ordinal $\theta>\kappa$ and every elementary submodel $X$ of $\VV_\theta$ of cardinality $\kappa$ with $\kappa+1\subseteq X$, there exists an ordinal $\vartheta>\kappa$ and an elementary embedding $\map{j}{X}{\VV_\vartheta}$ with $j\restriction\kappa=\id_\kappa$ and $j(\kappa)>\kappa$. 
  The results of  \cite{FuSa22}  show that all weakly extendible cardinals are weakly compact and therefore strongly inaccessible. 
   
Fuchino and Sakai prove that a cardinal $\kappa$ is a weak compactness cardinal for $\calL^2$ if and only if 
for every ordinal $\theta>\kappa$ and every elementary submodel $X$ of $\VV_\theta$ of cardinality $\kappa$ with $\kappa+1\subseteq X$,  there exists an ordinal $\vartheta>\kappa$ and an elementary embedding $\map{j}{X}{\VV_\vartheta}$ with $j(\kappa)>\sup(j[\kappa])$. 
 This equivalence directly implies that every  weakly extendible cardinal is a weak compactness cardinal for $\calL^2$. 
  Moreover, a standard application of Jensen's \emph{Covering Lemma} (as in {\cite[Theorem 18.27]{MR1940513}}) allows us to conclude  that the assumption $\VV=\LL$ causes all  weak compactness cardinals for $\calL^2$ to be weakly extendible. 
  In addition, a standard argument shows that, if $\kappa$ is the least  weak compactness cardinal for $\calL^2$ and $2^{{<}\kappa}=\kappa$ holds, then $\kappa$ is weakly extendible. 
  These implications directly raise the question of whether the existence of a  weak compactness cardinal for $\calL^2$ provably implies the existence of a weakly extendible cardinal. 
  A combination of Theorems \ref{MAIN:EssentiallyFaint} and \ref{MAIN:Equicons} now provides a negative answer to this question by showing that  weak compactness cardinals for $\calL^2$ can exist in models of $\ZFC$ without strongly inaccessible cardinals.   
\end{remark}

In the remainder of this section, we consider the question whether it is possible to extend the conclusion of Theorem \ref{MAIN:StrictLST}.\eqref{item:StrictLST1} to obtain a  characterization of the existence of  weak L\"owenheim--Skolem--Tarski numbers for all abstract logics through the assumption that $\Ord$ is essentially faint.

  \begin{theorem}\label{theorem:CharWeakLST}
  The following schemes are equivalent over $\ZFC$: 
   \begin{enumerate}
   \item\label{item:CharWeakLST1} Every abstract logic has a weak L\"owenheim--Skolem--Tarski number. 
   
   \item\label{item:CharWeakLST2} For every natural number $n$ and every cardinal $\eta$, there is a $C^{(n)}$-weakly shrewd cardinal $\kappa$ with $\delta^\eta<\kappa$ for all $\delta<\kappa$.  
   \end{enumerate}
 \end{theorem}
 
 \begin{proof}
   First, assume that \eqref{item:CharWeakLST1} holds, $n>0$ is a natural number and  $\mu$ is a cardinal. Pick a cardinal $\rho$ in $C^{(1)}$ above $\mu$.  Then there is an abstract logic $\calL$ that satisfies statements \eqref{item:CardinalFromLogic1}--\eqref{item:CardinalFromLogic5} listed in Lemma \ref{lemma:CardinalFromLogic} for $m=0$, $n$ and $\rho$ together with the following statements: 
   \begin{enumerate}
    \setcounter{enumi}{5}
     \item If $\tau$ is a language, then $\calL(\tau)$ contains the set $\calL_{\rho,\omega}(\tau)$ of all sentences of the infinitary logic $\calL_{\rho,\omega}$ that extends first-order logic by allowing conjunctions and disjunctions of less than $\rho$-many formulas. 
     
     \item Given  a language $\tau$, $M\in\Str_\tau$  and $\phi\in\calL_{\rho,\omega}(\tau)$, then $M\models_\calL\phi$ if and only if $M\models_{\calL_{\rho,\omega}}\phi$. 
   \end{enumerate}
 Our assumption now ensures that there is   a weak L\"owenheim--Skolem--Tarski number $\kappa$ for $\calL$.

 \begin{claim*}
  If $\delta<\kappa$, then $\delta^\eta<\kappa$.  
 \end{claim*}
 
 \begin{proof}[Proof of the Claim]
   Assume, towards a contradiction, that there exists a cardinal $\eta\leq\delta<\kappa$ with $\delta^\eta\geq\kappa$. Let $\tau$ denote a language that is obtained from $\sigma$ by first adding a constant symbol $c_\gamma$ for each $\gamma\leq\delta$ and then closing the language under Skolem functions for first order formulas.  
   Pick a sufficiently large cardinal $\theta>\kappa$ and a $\tau$-structure $N$ such that $\dom{N}$ is an elementary submodel of $\HH{\theta}$ of cardinality $\kappa$ with $\kappa+1\subseteq\dom{N}$, $E^N={\in}\restriction(\dom{N}\times\dom{N})$, $c_\gamma^N=\gamma$ for all $\gamma\leq\delta$ and all new function symbols are interpreted as the corresponding Skolem functions. 
   By elementarity, we know that the set $\dom{N}\cap{}^\eta\delta$ has cardinality $\kappa$. Since $\kappa$ is a  weak L\"owenheim--Skolem--Tarski number for $\calL$, we can find an $\calL$-elementary submodel $M$ of $N$ with $\betrag{\dom{M}}<\kappa$. 
   Then $\dom{M}$ is an elementary submodel of $\HH{\theta}$ with $\delta+1\subseteq\dom{M}$. Moreover, we know that there is a function $s\in N\cap{}^\eta\delta$ that is not an element of $M$. But, the set $\calL_{\rho,\omega}(\tau)$ contains a sentence $\phi_s$ that states that there is a function from $c_\eta$ to $c_\delta$ that sends $c_\gamma$ to $c_{s(\gamma)}$ for all $\gamma<\eta$. It then follows that $N\models_{\calL}\phi_s$ holds and $M\models_\calL\phi_s$ fails, contradicting the fact that $M$ is an $\calL$-elementary submodel  of $N$.  
 \end{proof}

 An application of Lemma \ref{lemma:CardinalFromLogic} now shows that the interval $(\rho,\kappa]$ contains a $C^{(n)}$-weakly shrewd cardinal $\mu$ with the property that $\delta^\eta<\mu$ holds for all $\delta<\mu$ with $\delta^\eta<\kappa$. By the above claim, we can conclude that $\delta^\eta<\mu$ holds for all $\delta<\mu$. This shows that \eqref{item:CharWeakLST2} holds in this case.

 Now, assume that \eqref{item:CharWeakLST2} holds. Pick an abstract logic $\calL$ and a natural number $n>0$ such that $\calL$ is definable by a $\Sigma_n$-formula with parameters. We can now apply \eqref{item:WeakCompactFromWShrewd5} of Lemma \ref{lemma:CnWeaklyShrewd} to see that every sufficiently large $C^{(n)}$-weakly shrewd cardinal $\kappa$ with the property that $\nu^{{<}\ooo_\calL}\leq\kappa$ holds for all $\nu<\kappa$     is a   weak L\"owenheim--Skolem--Tarski number for $\calL$.  Since our assumption ensures that there is a proper class of such cardinals, we can conclude that \eqref{item:CharWeakLST1} holds in this case.  
 \end{proof}

 The above result leaves open the question of whether the existence of weak L\"owenheim--Skolem--Tarski numbers for all abstract logics is provably equivalent to the assumption that $\Ord$ is essentially faint (see Question \ref{Q2} in the next section). 
 In the following, we will use Theorem \ref{theorem:CharWeakLST} to severely restrict the class of models of set theory in which the given principles are not equivalent. 
 For this purpose, recall that  the \emph{Singular Cardinal Hypothesis} $\SCH$ states that $\kappa^{\cof{\kappa}}=\kappa^+$ holds for every singular cardinal $\kappa$ with $2^{\cof{\kappa}}<\kappa$. In the following, we say that the \emph{$\SCH$  eventually holds} if this implication holds for eventually all singular cardinals.

 \begin{corollary}\label{corollary:SCHweakLST}
   The following schemes are equivalent over $\ZFC+\anf{\textit{$\SCH$ eventually holds}}$: 
   \begin{enumerate}
   \item\label{item:SCH1} Every abstract logic has a weak L\"owenheim--Skolem--Tarski number. 
   
   \item\label{item:SCH2} $\Ord$ is essentially faint. 
   \end{enumerate}
 \end{corollary}
 
 \begin{proof}
  In one direction, a combination of Theorem \ref{theorem:CharOrdFaint} with Theorem \ref{theorem:CharWeakLST} directly shows that \eqref{item:SCH1} implies \eqref{item:SCH2}. 
  In the other direction, note that, if the $\SCH$ eventually holds, then the proof of {\cite[Theorem 5.22]{MR1940513}} shows that for every cardinal $\mu$, all sufficiently large weakly inaccessible cardinals $\kappa$  have the property that $\nu^{{<}\mu}<\kappa$ holds for all $\nu<\kappa$. In particular, it follows that all sufficiently large $C^{(n)}$-weakly shrewd cardinals satisfy the assumptions of  \eqref{item:WeakCompactFromWShrewd5} in Lemma \ref{lemma:CnWeaklyShrewd} and therefore these cardinals are weak  L\"owenheim--Skolem--Tarski numbers for abstract logics definable by $\Sigma_n$-formulas. This shows that \eqref{item:SCH2} implies \eqref{item:SCH1}. 
 \end{proof}

 Remember that a classical result of Solovay in \cite{MR0379200} shows that the existence of a strongly compact cardinal implies that the $\SCH$ eventually holds. 
  In particular, the existence of such a cardinal entails that $\Ord$ is essentially faint   if and only if every abstract logic has a weak L\"owenheim--Skolem--Tarski number.  
 In another direction, Corollary \ref{corollary:SCHweakLST} also shows that the inequivalence of the two principles 
  implies the consistency of large cardinal assumptions in the region of measurable cardinals of high Mitchell order.
 This follows from the fact that this inequivalence requires a failure of the $\SCH$  at a proper class of singular cardinals 
 and such a  failure is known to imply the existence of such large cardinals  in canonical inner models (see, for example, \cite{MR1073778}).


\section{Open questions}

We end this paper by discussing two questions raised by the above results.

First, while the existence of strict L\"owenheim--Skolem--Tarski numbers for all abstract logics was shown to be equivalent to the assumption that $\Ord$ is essentially faint in Theorem \ref{MAIN:StrictLST}, it is unclear whether these principles are also equivalent to the existence of weak L\"owenheim--Skolem--Tarski numbers for all abstract logics. More specifically, Theorem \ref{theorem:CharWeakLST} leaves open the possibility that the equivalence stated in Theorem \ref{theorem:CharOrdFaint} can be strengthened so that the existence of $C^{(n)}$-weakly shrewd cardinals satisfying the cardinal arithmetic assumptions listed in Theorem \ref{theorem:CharWeakLST} can be derived from the   principle  \anf{\emph{$\Ord$ is essentially faint}}.

\begin{question}\label{Q2}
  Assume that $\Ord$ is essentially faint. Given a natural number $n>0$ and a cardinal $\eta$, is there is a $C^{(n)}$-weakly shrewd cardinal $\kappa$ with $\delta^\eta<\kappa$ for all $\delta<\kappa$?  
\end{question}

As discussed at the end of Section \ref{section:ProofsCompactness}, 
 a negative answer to Question \ref{Q2} would require substantial large cardinal assumptions. 
 However, it is not clear if the given fragment of the $\SCH$ outright follows from the assumption that $\Ord$ is essentially faint.

 Second, we ask whether the analysis carried out to separate the  principle \anf{\emph{$\Ord$ is essentially subtle}} from the principle \anf{\emph{$\Ord$ is essentially faint}} can be extended 
 to obtain characterizations of the consistency of stronger subtlety properties of the class of ordinals 
 through the existence and non-existence of  set-theoretic sentences with certain provable consequences.  
 More specifically, Theorem \ref{MAIN:SeparatePrinciples} shows that the consistency of the class-version of the large cardinal property of being a regular limit of subtle cardinals is equivalent to the existence of a consistent sentence implying that $\Ord$ is  essentially faint. 
 In addition, Theorem \ref{theorem:DiffSentence} provides a characterization of the consistency of the class-version of subtle limits of subtle cardinals through the non-existence of a sentence axiomatizing the principle  \anf{\emph{$\Ord$ is essentially subtle}} over the principle \anf{\emph{$\Ord$ is essentially faint}}.  
These results raise the question whether the consistency of the class-version of the next stronger large cardinal property, {i.e.,} the property of being a stationary limit of subtle cardinals, can be characterized in a similar manner. 

\begin{question}
 Are there analogs of Theorems  \ref{MAIN:SeparatePrinciples} and \ref{theorem:DiffSentence} that  characterize the consistency of the theory $$\ZFC ~ + ~ \anf{\textit{The class of subtle cardinals is stationary in $\Ord$}}?$$  
\end{question}


\bibliographystyle{amsplain} 
\bibliography{biblio}

\providecommand{\bysame}{\leavevmode\hbox to3em{\hrulefill}\thinspace}
\providecommand{\MR}{\relax\ifhmode\unskip\space\fi MR }
\providecommand{\MRhref}[2]{%
  \href{http://www.ams.org/mathscinet-getitem?mr=#1}{#2}
}
\providecommand{\href}[2]{#2}
\begin{thebibliography}{10}

\bibitem{zbMATH06029795}
Joan Bagaria, \emph{{$C^{(n)}$}-cardinals}, Arch. Math. Logic \textbf{51}
  (2012), no.~3-4, 213--240.

\bibitem{Patterns}
Joan Bagaria and Philipp L{\"u}cke, \emph{Patterns of {S}tructural {R}eflection
  in the large-cardinal hierarchy}, Trans. Amer. Math. Soc. \textbf{377}
  (2024), no.~5, 3397--3447.

\bibitem{MR3519447}
Joan Bagaria and Jouko V\"{a}\"{a}n\"{a}nen, \emph{On the symbiosis between
  model-theoretic and set-theoretic properties of large cardinals}, J. Symb.
  Log. \textbf{81} (2016), no.~2, 584--604.

\bibitem{MR4093885}
Will Boney, \emph{Model theoretic characterizations of large cardinals}, Israel
  J. Math. \textbf{236} (2020), no.~1, 133--181.

\bibitem{bdgm}
Will Boney, Stamatis Dimopoulos, Victoria Gitman, and Menachem Magidor,
  \emph{Model theoretic characterizations of large cardinals revisited}, Trans.
  Amer. Math. Soc. \textbf{377} (2024), no.~10, 6827--6861.

\bibitem{MR1059055}
Chen~C. Chang and H.~Jerome Keisler, \emph{Model theory}, third ed., Studies in
  Logic and the Foundations of Mathematics, vol.~73, North-Holland Publishing
  Co., Amsterdam, 1990.

\bibitem{MR750828}
Keith~J. Devlin, \emph{Constructibility}, Perspectives in Mathematical Logic,
  Springer-Verlag, Berlin, 1984.

\bibitem{FuSa22}
Saka\'e{} Fuchino and Hiroshi Sakai, \emph{{Weakly extendible cardinals and
  compactness of extended logics}}, preprint, 2022.

\bibitem{outwardcompactness}
Peter Holy, Philipp L{\"u}cke, and Sandra M{\"u}ller, \emph{Outward
  compactness}, to appear in the {\emph{Israel Journal of Mathematics}}, 2025.

\bibitem{MR1940513}
Thomas Jech, \emph{Set theory}, Springer Monographs in Mathematics,
  Springer-Verlag, Berlin, 2003, The third millennium edition, revised and
  expanded.

\bibitem{jensennotes}
Ronald~B. Jensen and Kenneth Kunen, \emph{Some combinatorial properties of
  {$L$} and {$V$}}, handwritten notes, 1969.

\bibitem{luecke2021strong}
Philipp L\"{u}cke, \emph{Strong unfoldability, shrewdness and combinatorial
  consequences}, Proc. Amer. Math. Soc. \textbf{150} (2022), no.~9, 4005--4020.

\bibitem{SRminus}
\bysame, \emph{Structural reflection, shrewd cardinals and the size of the
  continuum}, J. Math. Log. \textbf{22} (2022), no.~2, Paper No. 2250007, 43.

\bibitem{MR295904}
Menachem Magidor, \emph{On the role of supercompact and extendible cardinals in
  logic}, Israel J. Math. \textbf{10} (1971), 147--157.

\bibitem{MR2833152}
Menachem Magidor and Jouko V\"a\"an\"anen, \emph{On
  {L}\"owenheim-{S}kolem-{T}arski numbers for extensions of first order logic},
  J. Math. Log. \textbf{11} (2011), no.~1, 87--113.

\bibitem{MR780522}
Johann~A. Makowsky, \emph{Vop\v{e}nka's principle and compact logics}, J.
  Symbolic Logic \textbf{50} (1985), no.~1, 42--48.

\bibitem{MR930262}
Pierre Matet, \emph{Some aspects of {$n$}-subtlety}, Z. Math. Logik Grundlag.
  Math. \textbf{34} (1988), no.~2, 189--192.

\bibitem{MR1073778}
William~J. Mitchell, \emph{On the singular cardinal hypothesis}, Trans. Amer.
  Math. Soc. \textbf{329} (1992), no.~2, 507--530.

\bibitem{MR1369172}
Michael Rathjen, \emph{Recent advances in ordinal analysis: {$\Pi^1_2$}-{${\rm
  CA}$} and related systems}, Bull. Symbolic Logic \textbf{1} (1995), no.~4,
  468--485.

\bibitem{MR0379200}
Robert~M. Solovay, \emph{Strongly compact cardinals and the {GCH}}, Proceedings
  of the {T}arski {S}ymposium, Proc. Sympos. Pure Math., vol. Vol. XXV,
  Published for the Association for Symbolic Logic by the American Mathematical
  Society, Providence, RI, 1974, pp.~365--372.

\bibitem{MR482431}
Robert~M. Solovay, William~N. Reinhardt, and Akihiro Kanamori, \emph{Strong
  axioms of infinity and elementary embeddings}, Ann. Math. Logic \textbf{13}
  (1978), no.~1, 73--116.

\bibitem{MR1649079}
Andr\'{e}s Villaveces, \emph{Chains of end elementary extensions of models of
  set theory}, J. Symbolic Logic \textbf{63} (1998), no.~3, 1116--1136.

\end{thebibliography}

\end{document}